\documentclass[reqno,a4paper]{siamart171218}

\pdfpageattr{
  /CropBox [40 120 475 785]
}

\usepackage{amsmath,amssymb,amsfonts,mathtools}
\usepackage{tabularx}
\usepackage{multirow}        
\usepackage{diagbox}
\usepackage{tikz}
\usepackage{algorithm}
\usepackage{algpseudocode}
\usepackage{graphicx,epsfig,color,scalerel,subfigure,booktabs}
\usepackage[super]{nth}
\usepackage{xparse}
\usepackage{hyperref}
\usepackage{siunitx}
\usepackage{afterpage}
\usepackage{enumitem}
\usepackage{comment}
\usepackage{stmaryrd}
\usepackage[normalem]{ulem}
\definecolor{lightgreen}{rgb}{0.22,0.50,0.25}
\definecolor{lightblue}{rgb}{0.22,0.45,0.70}

\newtheorem{remark}{Remark}
\newcommand{\bn}{\boldsymbol{n}}

\newcommand{\bu}{\boldsymbol{u}}
\newcommand{\bv}{\boldsymbol{v}}
\newcommand{\bw}{\boldsymbol{w}}

\renewcommand{\bf}{\boldsymbol{f}}

\newcommand{\bx}{\boldsymbol{x}}

\newcommand{\bz}{\boldsymbol{z}}

\newcommand\bzero{\boldsymbol{0}}
\newcommand\bkappa{\boldsymbol{\kappa}}

\newcommand{\rF}{\mathtt{F}}

\newcommand{\rH}{\mathrm{H}}
\newcommand{\rL}{\mathrm{L}}
\newcommand{\rW}{\mathrm{W}}
\newcommand{\rQ}{\mathrm{Q}}

\newcommand{\bH}{\mathbf{H}}

\newcommand{\bR}{\mathbf{R}}
\newcommand{\bL}{\mathbf{L}}
\newcommand{\bW}{\mathbf{W}}
\newcommand{\bV}{\mathbf{V}}

\newcommand{\bZ}{\mathbf{Z}}

\newcommand{\cA}{\mathcal{A}}
\newcommand{\cB}{\mathcal{B}}

\newcommand{\cE}{\mathcal{E}}
\newcommand{\cL}{\mathcal{L}}
\newcommand{\cT}{\mathcal{T}}
\newcommand{\cS}{\mathcal{S}}

\newcommand\vdiv{\mathop{\mathrm{div}}\nolimits}

\numberwithin{equation}{section}
\numberwithin{remark}{section}
\numberwithin{figure}{section}
\numberwithin{table}{section}

\definecolor{forestgreen}{RGB}{0,150,0}

\newcommand{\revthree}{}
\newcommand{\revone}{}
\newcommand{\revtwo}{}

\allowdisplaybreaks

\headers{Robust methods for Darcy--Forchheimer equations}{Das, Hutridurga, Pani, Ruiz-Baier}
\title{Robust stability and preconditioning of Darcy--Forchheimer equations\thanks{\textbf{Updated:} \today.\funding{This work has been partially supported by  the Australian Research Council through the \textsc{Future Fellowship} grant FT220100496 and \textsc{Discovery Project} grant DP22010316; and by the Centre for Advanced Study (CAS) at the Norwegian Academy of Science and Letters under the program \textsc{Mathematical Challenges in Brain Mechanics}.}}}

\author{Rishi Das\thanks{
IITB–Monash Research Academy, Indian Institute of Technology Bombay, Powai, Mumbai, Maharashtra 400076, India (\email{rishi.das@iitb.ac.in}).} 
\and  Harsha Hutridurga\thanks{Department of Mathematics, Indian Institute of Technology Bombay, Powai, Mumbai, Maharashtra 400076, India (\email{harsha@math.iitb.ac.in}).}
\and Amiya K. Pani\thanks{Department of Mathematics, BITS-Pilani, KK Birla Goa Campus, Zuarinagar, Goa-403726, India (\email{amiyap@goa.bits-pilani.ac.in}).}
\and Ricardo Ruiz-Baier\thanks{Corresponding author: School of Mathematics, Monash University, 9 Rainforest Walk, Melbourne VIC 3800, Australia, and Universidad Adventista de Chile, Casilla 7-D Chill\'an, Chile (\email{ricardo.ruizbaier@monash.edu}).}}
\date{\today}

\begin{document}

\maketitle

\begin{abstract}
    We derive parameter-robust quasi-optimal error estimates for  mixed finite element methods for the nonlinear Darcy--Forchheimer equations with mixed boundary conditions. Using the framework of operator preconditioning, we also design efficient block preconditioners for the linearised system, that exhibit robustness with respect to the coefficients that modulate permeability and inertia of the system. The properties of the formulation (parameter and mesh-size independence of the convergence rates) are illustrated by means of several numerical examples. 
\end{abstract}

\section{Introduction}
\paragraph{Scope} 
The Darcy--Forchheimer equations model fluid flow through porous media, when both viscous and inertial effects are significant. While Darcy's law describes the linear flow proportional to the pressure gradient and is valid for slow, creeping flows, the Forchheimer extension introduces a nonlinear correction to account for inertial effects (manifesting itself as term depending as a power law on the velocity magnitude, taking into account kinematic energy loss) that arise at moderate flow velocities. These equations have broad applications in engineering and geosciences (high-velocity subsurface flows such as near injection or extraction wells in gas reservoirs and conglomerate-confined aquifers,  contaminant transport in environmental engineering, and in filtration and catalytic bed reactors in chemical engineering) \cite{liu2024experimental,mathias2014heat}, as well as  in biological systems (e.g., blood perfusion in tissues, where nonlinearity in flow resistance can be significant) \cite{grillo2014darcy,hira2025advanced}. 

Mixed finite element methods provide a natural framework for the numerical approximation of Darcy--Forchheimer flows, as they offer local mass conservation and accurate flux approximations. Their construction and analysis have been addressed through fixed-point and Newton methods \cite{al2024iterative,audu2019mixed,lopez2009comparison,pan2012mixed,girault2008numerical,huang2018multigrid,caucao2020conforming}, often coupled with iterative solvers tailored to the nonlinear structure. 

For Darcy and Forchheimer-type equations, which model creeping and inertial flows, respectively, the associated saddle-point systems exhibit strong sensitivity to variations in physical parameters such as permeability and the Forchheimer coefficient. Robust iterative solvers require preconditioners that maintain spectral equivalence independently of these parameters. For this, we follow the framework of operator preconditioning from \cite{mardal2011preconditioning} (see, also the earlier works \cite{klawonn1995optimal,arnold1997preconditioning}), which suggests linking the preconditioning of algebraic systems with their infinite-dimensional operator counterpart. This allows to construct preconditioners by approximating the Riesz maps of the continuous operators in appropriate function space norms (or duality maps in the case of Banach spaces). For the mixed formulation of the Darcy problem, this approach yields block-diagonal preconditioners involving sums of spaces of pressure, as developed in \cite{baerland2020observation} (and also extended for Biot equations and other families of perturbed saddle-point problems in, e.g.,  \cite{boon2021robust, boon22, cimr_nmpde24, hong2019conservative}). However, these works do not address in details the error analysis, which is also expected to maintain parameter robustness. 

The objective of this work is twofold. First, to specify C\'ea-type estimates and convergence rates using the parameter-weighted norms, and secondly, to apply this approach to a (relatively simple) nonlinear extension to Darcy--Forchheimer flow equations.  Obtaining uniform stability with respect to model parameters is of great interest for multiphysics coupling models since the variation of coefficients across sub-domains, or spatial scales, or different applications, can be extremely large. Also, coupling through interfaces with other flow regimes typically implies that we cannot simply rescale one equation, but certain compatibility of the parametrisation of both sub-models must be maintained. 

While operator preconditioning techniques are available for many linear elliptic and saddle point problems, developing such robust solvers for general nonlinear saddle point problems exhibits significant challenges. Available techniques include multilevel, domain decomposition, operator preconditioning in the Hilbert context, field-split algorithms, and Sobolev gradient methods, for example (see, the review in \cite{farago2002numerical}). In the nonlinear Forchheimer regime, linearisation techniques (e.g., Picard or Newton) give rise to a sequence of linearised problems, for which an operator preconditioning approach can be extended using variable-coefficient norm scalings following \cite{schoberl2007symmetric}. Inspired by \cite{piersanti2021parameter} for multiphysics couplings, using the norms induced by the parameter-robust stability analysis, this paper proposes and analyse  simple  block diagonal (variable) operator preconditioners for the linearised Darcy--Forchheimer system.  

In order to deal with intersections and sums of parameter-weighted function spaces, from \cite{mardal2011preconditioning,boon2021robust}, it is possible to prove continuous and discrete inf-sup conditions by carefully combining auxiliary PDEs with different regularity and parameter weighting. This approach also requires a lifting map, and then, appropriate maps are constructed using the aforementioned auxiliary problems. Based on similar techniques, here we first discuss well-posedness of three different suitable PDEs and secondly, derive the inf-sup stability using a lifting map in the discrete case.

\paragraph{Outline} The  paper is organised as follows. Section~\ref{sec:model} recalls the necessary notations, the boundary value problem and the derivation of the weak formulation using parameter-weighted norms. Section~\ref{sec:continuous} addresses the well-posedness of the weak formulation robustly with respect to model parameters, following the abstract theory of Minty--Browder monotone operators. In  Section~\ref{sec:discrete}, we discuss the  robust solvability of the discrete problem. In addition, we present quasi-optimal error analysis  and derive precise convergence rates for conforming discretisations of the nonlinear model in Section \ref{sec:apriori}.  Next, in Section~\ref{sec:precond} we discuss how a simple variable operator preconditioning can be used in this context of Banach spaces, and we conclude in Section~\ref{sec:numerics} with a set of numerical tests illustrating the properties of the proposed schemes. 

\section{Model problem}\label{sec:model}
\paragraph{Notation and preliminaries}
Let $\Omega\subset \mathbb{R}^d$, $d\in \{2,3\}$, denote a  {bounded} domain with Lipschitz-continuous boundary $\partial\Omega$, on which its outward unit normal  is
denoted by $\bn$. For $s\geq 0$ and \revone{$t\in[1,+\infty]$}, we denote by \revone{$\rL^t(\Omega)$ and $\rW^{s,t}(\Omega)$} the usual Lebesgue and Sobolev spaces endowed with the norms \revone{$\|\bullet\|_{0,t,\Omega}$ and $\|\bullet\|_{s,t,\Omega}$, respectively.
Note that $\rW^{0,t}(\Omega)=\rL^t(\Omega)$. 
If $t = 2$}, we write $\rH^{s}(\Omega)$ in place of $\rW^{s,2}(\Omega)$, and denote the corresponding norm by $\|\bullet\|_{s,\Omega}$. We 
denote by $\mathbf{H}$ the   vectorial counterpart of a generic scalar functional space $\rH$. The $\rL^2(\Omega)$ inner product for scalar, vector, or tensor valued functions
is denoted by $(\bullet,\bullet)_{0,\Omega}$. We also recall the Hilbert space $\bH(\vdiv,\Omega)=\{ \bv \in \bL^2(\Omega): \vdiv \bv \in \rL^2(\Omega)\}$ endowed with the norm $\|\bv\|_{\vdiv,\Omega}^2 = \|\bv\|^2_{0,\Omega} + \|\vdiv \bv\|^2_{0,\Omega}$. Moreover, we  use for $t,s \in [1,+\infty)$, the Banach space 
$$\revone{\bH^t(\vdiv_s,\Omega)=\{ \bv \in \bL^t(\Omega): \vdiv \bv \in \rL^s(\Omega)\},}$$ endowed with the norm $\|\bv\|_{t,\vdiv_s,\Omega}^2 = \|\bv\|^t_{0,t,\Omega} + \|\vdiv \bv\|^s_{0,s,\Omega}$. 
Often, we  use a subscript $\Gamma_i$ on the functional space to denote that the (appropriate) traces vanish on the part of the boundary $\Gamma_i\subset \partial \Omega$. 

For two Banach spaces $(X,\|\bullet\|_X)$ and $(Y,\|\bullet\|_Y),$ we denote by $\cL(X,Y)$ the space of bounded linear maps from $X$ to $Y$. Let $(Z,\|\bullet\|_Z)$ be another Banach space. The intersection and sum spaces are 
\begin{gather*} X \cap Y \cap Z = \{ v : v\in X \ \text{and} \ v\in Y \ \text{and}\ v\in Z\}, \\ X+ Y + Z = \{u+v+w: u\in X \ \text{and}\ v\in Y\ \text{and} \ w \in Z\},\end{gather*}
and their respective norms are (see \cite{baerland2020observation}, noting that the same norm characterisation done there for Hilbert spaces holds also for the Banach case)
\[ \| r\|^2_{X\cap Y \cap Z} = \|r\|^2_X + \|r\|^2_Y + \|r\|^2_Z, \quad \| z\|^2_{X + Y + Z} = \inf_{\substack{r = u+v+w, \\ u \in X, v\in Y, w\in Z}}\,\bigl[\|u\|^2_X + \|v\|^2_Y+ \|w\|^2_Z\bigr],\]
\revone{and we also use the convenient equivalent form 
\[ \| r\|_{X\cap Y \cap Z} = \|r\|_X + \|r\|_Y + \|r\|_Z, \quad \| z\|_{X + Y + Z} = \inf_{\substack{r = u+v+w, \\ u \in X, v\in Y, w\in Z}}\,\bigl[\|u\|_X + \|v\|_Y+ \|w\|_Z\bigr].\]
Furthermore}, if $X \cap Y \cap Z$ is dense in $X$ and $Y$ and $Z$, then following relation (understood as an identification under an isometry) holds 
\begin{equation}\label{eq:duality} (X+Y+Z)' = X'\cap Y' \cap Z'.\end{equation}
In addition, for a fixed positive constant $\alpha$, by $\alpha X$ we denote the Banach space whose elements coincide with those in $X$ but are measured with the norm 
\[ \| \bullet \|^2_{\alpha X}= \alpha^2\|\bullet \|^2_X. \]

\paragraph{Strong form of the governing equations}
We assume that the domain boundary $\partial\Omega$ is decomposed between two sub-boundaries $\Gamma_{\bu}$ and $\Gamma_p$, both with positive $(d-1)-$Hausdorff measure. The governing equations consist in finding the discharge flux $\bu$ and pressure head $p$ satisfying 
\begin{subequations}\label{eq:strong}
\begin{align}
    \kappa^{-1}\bu + \rF|\bu|^{r-2}\bu + \nabla p & = \bf \quad \text{in $\Omega$},\\
    \vdiv \bu & = g \quad \text{in $\Omega$},\\
    \bu\cdot \bn & = 0 \quad \text{on $\Gamma_{\bu}$},\\
    p & = 0 \quad \text{on $\Gamma_p$}.
\end{align}\end{subequations}
Here, $\kappa$ is the hydraulic conductivity (permeability of the porous matrix divided by the fluid viscosity), $\bf$ is a forcing vector, $g$ is a fluid source, and $\rF>0$, $r\in [3,4]$ are the Forchheimer coefficient and index, respectively. 

Regarding the weak formulation of \eqref{eq:strong}, we now consider generic spaces $\bV$ and $\rQ$ for velocity and pressure, respectively (and to be made precise below), yielding:  
Find $(\bu,p)\in \bV\times \rQ$ such that 
\begin{subequations}\label{eq:weak}
    \begin{align}
    a(\bu,\bv) + c(\bu; \bu,\bv) +  b(\bv,p) & = F(\bv) \quad \forall \bv \in \bV,\label{eq:weak,1}\\
    b(\bu,q) & = G(q) \quad \forall q \in \rQ,\label{eq:weak,2}
    \end{align}
\end{subequations}
where for sufficiently regular $\bV$ (incorporating the flux boundary condition in an essential manner) and $\rQ,$ the following bilinear forms $a:\bV\times\bV\to\mathbb{R}$, $b:\bV\times\rQ \to \mathbb{R}$,  linear functionals $F:\bV\to\mathbb{R}$, $G:\rQ\to\mathbb{R}$, and nonlinear form $c:\bV\times\bV\times\bV \to \mathbb{R}$, are defined, respectively, as: 
\begin{gather*}
  a(\bu,\bv):= \int_\Omega \kappa^{-1}\bu\cdot\bv, \quad b(\bv,q):= \revtwo{-\int_\Omega q\,\vdiv\bv}, \quad F(\bv):= \int_\Omega \bf\cdot\bv, \\ G(q):= \revtwo{-\int_\Omega g\,q}, \quad c(\bw;\bu,\bv):=\int_\Omega \rF |\bw|^{r-2}\bu\cdot\bv.   
\end{gather*}

\section{Continuous robust solvability}\label{sec:continuous}
First, we recall  an abstract setting and its solvability using  the Minty--Browder framework, \revone{which  we  use for} the well-posedness of \eqref{eq:weak}. A proof can be found in, e.g.,  \cite[Theorem 3.1]{caucao2023mixed} (see also  \cite[Theorem 3.1]{caucao2020conforming}). 
\begin{theorem}[Abstract setting for well-posedness]\label{th:abstract}
Let $X,Y$ be two reflexive Banach spaces, $\cA:X\to X'$ a nonlinear map, $\cB:X\to Y'$ a linear and bounded operator, and denote by 
 $Z:= \{ v\in X: \cB(v) = 0\}$ the kernel of $\cB$. Assume that 
 \begin{itemize}
     \item (Local Lipschitz continuity). There exist $\gamma>0$ and $s\geq 2$ such that 
     \[ \|\cA(u)-\cA(v)\|_{X'} \leq \gamma \|u-v\|_X + \gamma \|u-v\|_X(\|u\|_X+\|v\|_X)^{s-2} \qquad \forall u,v\in X. \]
     \item ({Uniformly strong monotonicity} on the kernel). For a fixed $w \in X$, {there exists $\alpha>0$ such that} 
     \[ \langle \cA(u+w) - \cA(v+w),u-v\rangle \geq \alpha \|u-v\|_X^{2}\qquad \forall u,v\in Z. \]
     \item (Inf-sup stability). There exists $\beta>0$ such that
     \[ \sup_{\substack{v \in X \\ v \neq 0}} \frac{\langle \cB(v), q\rangle}{\| v \|_X} \geq \beta \| q \|_Y \quad \forall q \in Y.\]
 \end{itemize}
 Then, for each $(f, g) \in X' \times Y'$ there exists a unique $(u, p) \in X \times Y$ such that
\begin{align*}
\langle \cA(u), v\rangle + \langle \cB(v), p\rangle &= \langle f, v\rangle \quad \forall v \in X, \\
\langle \cB(u), q\rangle &= \langle g, q\rangle \quad \forall q \in Y.
\end{align*}

Moreover, there exists $C > 0$, depending only on $\alpha, \gamma, \beta, \text{and } s$ such that
\begin{equation}
{\| (u, p) \|_{X \times Y} \leq C \bigl(\mathcal{M}(f, g)+ \revone{\bigl(\mathcal{M}(f, g)\bigr)}^{s-1}\bigr)},
\end{equation}
where
$\mathcal{M}(f, g) := \| f \|_{X'} + \| g \|_{Y'} + \| g \|_{Y'}^{s - 1} + \| \cA(0) \|_{X'}$.
\end{theorem}

Theorem~\ref{th:abstract} readily yields the unique solvability of \eqref{eq:weak} if we specify the spaces $\bV = \bH^{\revone{r}}_{\Gamma_{\bu}}(\vdiv,\Omega)$ and $\rQ = \rL^2(\Omega)$. This has been established in, e.g., \cite{pan2012mixed} \revone{for the case $r=3$}. However, the properties of the weak formulation (in particular, {uniformly strong monotonicity} on the kernel and inf-sup condition) are not necessarily uniform with respect to the model parameters, leading to lack of parameter robustness of the corresponding continuous dependence on data. 

\revtwo{For a robust well-posedness of the Darcy--Forchheimer formulation, we consider the following  weighted spaces 
\[\bV:=  \kappa^{-\frac{1}{2}}\bL^2(\Omega) \cap \bH_{\Gamma_{\bu}}(\vdiv,\Omega) \cap \revone{\rF^{\frac1r} \bL^r(\Omega)}, \quad \rQ := \rL^2(\Omega) + \kappa^{\frac{1}{2}}\rH^1_{\Gamma_p}(\Omega) + \revone{\rF^{-\frac1r}\rW^{1,r'}_{\Gamma_p}(\Omega)},\]
endowed with the following norms 
\begin{subequations}\label{eq:norms}
    \begin{align}
   \|\bv\|_{\bV} &:= \kappa^{-\frac12}\|\bv\|_{0,\Omega} + \|\vdiv\bv\|_{0,\Omega} + \revone{\rF^{\frac1r}\|\bv\|_{0,r,\Omega}},\label{eq:norms-V}\\
   \|q\|_{\rQ} &:= \!\!\inf_{\substack{q = q_1+q_2+q_3, \\ (q_1,q_2,q_3) \in \rL^2(\Omega) \times \rH^1_{\Gamma_{\!p}}\!(\Omega)\times \revone{\rW^{1,r'}_{\Gamma_p}(\Omega)}}}\!\!\bigl[\|q_1\|_{0,\Omega} + \kappa^{\frac{1}{2}}\|q_2\|_{1,\Omega} + \revone{\rF^{-\frac1r} |q_3|_{1,r',\Omega}}\bigr],\label{norm:Q}
    \end{align}
\end{subequations}
respectively. This choice becomes evident in the analysis to follow.}
We adapt the operator preconditioning approach from  \cite{mardal2011preconditioning} (and used for Darcy equations in \cite{baerland2020observation}) and proceed first to show an auxiliary property of the divergence operator that allows us to prove a parameter-robust inf-sup condition.
\begin{lemma}\label{lem:inf-sup}
 There exists an operator \(\cS_c\) satisfying the following properties:
    \begin{gather}\nonumber
    \cS_c \in \cL(\rQ', \bV),\quad 
    \|\cS_c\|_{\cL(\rQ', \bV)}\ \text{is independent of } \kappa,\rF,\quad \text{and}\\
    ({\vdiv} [\cS_c(g)], \psi) = \langle g, \psi \rangle \quad \text{for all}\ g \in \rQ',\psi \in \rQ.\label{S:prop}
\end{gather}
\end{lemma}
\begin{proof}
We proceed to define $\cS_c$ as a bounded linear operator from three different spaces, and then, invoke the duality relation between sum and intersection between Banach spaces \eqref{eq:duality} to conclude that \( \cS_c \in {\cL}(\rQ', \bV) \).

\revone{Recall that $r\in[3,4]$, and let $r'$ be such $\frac{1}{r}+\frac{1}{r'}=1$}. 
For the existence of \( {\cS_c} \in {\cL}\bigl((\rW_{\Gamma_p}^{1, q}(\Omega))', {\bL}^{\revone{q'}}(\Omega) \bigr) \)
with $q\in \{2,\;\revone{r'}\}$ \revone{and $\frac{1}{q} + \frac{1}{q'}=1$,} let  
\( \phi \in \rW^{1,\revone{q'}}_{\Gamma_p}(\Omega)\)  for a given  \( g \in (\rW_{\Gamma_p}^{1,q}(\Omega))'\)  be \revone{the} unique weak solution of the following elliptic problem with mixed boundary conditions (owing to, e.g., \cite[Chapter 7]{maz_i_a2010elliptic}) 
\begin{align}\label{eq:elliptic reg}
-\Delta \phi = g \quad \text{in} \ \Omega, \quad \nabla \phi \cdot {\bn} = 0 \quad \text{on} \ \Gamma_{\bu}, \quad \phi = 0 \quad \text{on } \Gamma_p.
\end{align}
Then, it  suffices to define \( {\cS}_c(g) := \nabla \phi \), 
where $\nabla \phi \in \bL^{\revone{q'}}(\Omega)$.  This, in turn, shows
\(
{\cS_c} \in {\cL}\bigl( (\rW^{1,q}_{\Gamma_p}(\Omega))', {\bL}^{\revone{q'}}(\Omega) \bigr)\), for $q=2$ and $q= \revone{r'}$ corresponding to \revone{$q'=2$ and $q'=r$}, respectively.

Now, for   
\(g\in \left( \rL^2(\Omega)\right)',\) let \((\bw,r)\in \mathbf{H}_{\Gamma_{\bu}}({\vdiv},\Omega)\times \rL^2(\Omega)\) be the unique solution (cf.  \cite[Section 7.1.2]{boffi2013mixed}) of 
\begin{subequations}\label{eq:H(div)-L2 formulation}
\begin{align}
(\bw,\bz)+({\vdiv}\bz,r)&=0 \quad &&\forall \bz\in \mathbf{H}_{\Gamma_{\bu}}({\vdiv},\Omega), \\
({\vdiv}\bw,s)&=\langle g,s\rangle \quad&&\forall s\in \rL^2(\Omega).
\end{align}
\end{subequations}
Then, we note that $\cS_c$ can also be defined as \({\cS_c}(g):=\bw\), and therefore,  
\begin{equation}\label{eq:S-2}
\mathcal{S}_c \in {\cL}\bigl( \rL^2(\Omega), {\bH}_{\Gamma_{\bu}}({\vdiv}, \Omega) \bigr). 
\end{equation}
This completes the proof.
\end{proof}

\begin{theorem}[Parameter-robust well-posedness]\label{th:continuous}
 Assume that $\bf\in \bL^2(\Omega),g\in \rL^2(\Omega).$ Then, there exists a unique   $(\bu,p)\in \bV\times\rQ$ \revone{solution to} \eqref{eq:weak}. Furthermore, there is a positive constant $\tilde{C},$ independent of the model parameters $\kappa, \rF$, such that 
\begin{equation*}
 \|(\bu, p )\|_{\bV\times\rQ}  \leq \tilde{C}\max \bigl\{\|\bf\|_{0,\Omega} + \| g \|_{0,\Omega} + \revone{\|g\|^{r-1}_{0,\Omega}},(\|\bf\|_{0,\Omega} + \| g \|_{0,\Omega} + \revone{\|g\|^{r-1}_{0,\Omega}})^2\bigr\}.
    \end{equation*}
\end{theorem}
\begin{proof}
In order to apply Theorem \ref{th:abstract}, we define the operators $A:\bV\to \bV'$, $B:\bV\to\rQ'$, and $C:\bV\to \bV'$ as \(\langle A\bu,\bw\rangle:=a(\bu,\bw)\), \(\langle B\bv,p\rangle:=b(\bv,p)\), and \(\langle C(\bu),\bw\rangle := c(\bu;\bu,\bw)\), respectively. Now, set  the non-linear operator $A+C:\bV\to\bV'$. We proceed to verify the assumptions of Theorem \ref{th:abstract} identifying $X \equiv \bV$, $Y \equiv \rQ$, $\cA \equiv A+C$, $\cB \equiv B$, and \revone{$s = r$}.

\medskip 
\paragraph{\it Local Lipschitz continuity} For \(\bu,\bv,\bw\in \bV,\) apply H\"older's inequality to obtain 
        \begin{align*}
   & \langle \{A+C\}({\bu}),\bw\rangle - \langle \{A+C\}({\bv}),\bw\rangle \\
   &
     = \int_{\Omega}\kappa^{-1}(\bu-\bv)\cdot\bw + \int_{\Omega}\rF(\revone{|\bu|^{r-2}}\bu-\revone{|\bv|^{r-2}}\bv)\cdot\bw\\
     &\quad \leq \kappa^{-\frac12} \|{\bu} - {\bv}\|_{0,\Omega} \kappa^{-\frac12} \|\bw\|_{0,\Omega} + \revone{\rF^{\frac1r} \|\bw\|_{0,r,\Omega} 
      \rF^{\frac{1}{r'}} \| |\bu|^{r-2}\bu - |\bv|^{r-2}\bv \|_{0,r',\Omega}}. 
\end{align*}
Using \cite[Lemma 2.1, equation \((2.1a)\)]{barrett1993finite} to bound the second term on the right-hand side of the above expression, we deduce that there exists a positive constant \(\tilde{C}\) independent of  \(\kappa, \rF,\) such that
\begin{align*}
  &  \langle \{A+C\}({\bu}),\bw\rangle - \langle \{A+C\}({\bv}),\bw\rangle \\
  & \leq \kappa^{-\frac12} \|{\bu} - {\bv}\|_{0,\Omega} \kappa^{-\frac12} \|\bw\|_{0,\Omega}  + \revone{\rF^{\frac1r} \|\bw\|_{0,r,\Omega} 
      \tilde{C}\rF^{\frac1r}  (\|\bu\|_{0,r,\Omega} + \|\bv\|_{0,r,\Omega})\rF^{\frac1r}\|\bu-\bv\|_{0,r,\Omega}} \\
      &\leq \|\bu-\bv\|_{\bV}\|\bw\|_{\bV}+\tilde{C}(\|\bu\|_{\bV} + \|\bv\|_{\bV})\|\bu-\bv\|_{\bV}\|\bw\|_{\bV},
\end{align*}
and therefore, there holds 
\begin{equation}\label{eq:lip}\|\{A+C\}({\bu}) - \{A+C\}({\bv})\|_{\bV'}\leq \|\bu-\bv\|_{\bV}+\tilde{C}(\|\bu\|_{\bV} + \|\bv\|_{\bV})\|\bu-\bv\|_{\bV}. \end{equation}

\paragraph{\it {Uniformly strong monotonicity} of \(\{A+C\}(\bullet+\bw)\)} First we stress that for all $\bu\in \bZ:=\{\bv \in \bV: b(\bv,q)=0 \quad \forall q\in \rQ\},$ we can choose $q=\vdiv \bu \in \rL^2(\Omega)\subseteq \rQ$ and therefore, we have the characterisation 
$ \bZ= \{ \bv\in \bV: \vdiv\bv = 0\}$. 
Now, for $\bu,\bv \in \bZ$  and $\bw\in \bV$, it holds 
     \begin{align*}
        &\langle \{A+C\}(\bu+\bw)-\{A+C\}(\bv+\bw),\bu-\bv\rangle\\ &\quad =\int_\Omega  \kappa^{-1}|\bu-\bv|^2+
        \rF\int_\Omega   \revone{\left((\bu+\bw)|\bu+\bw|^{r-2}-(\bv+\bw)|\bv+\bw|^{r-2}\right)}\cdot(\bu-\bv).
     \end{align*}
By employing \cite[Lemma 2.1 and (2.1b)]{barrett1993finite}, we deduce that there exists a constant \(\tilde{C}>0\) depending only on the volume of \(\Omega\), such that 
     \begin{align*}
         \rF\int_\Omega  \left((\bu+\bw)|\bu+\bw|^{r-2}-(\bv+\bw)|\bv+\bw|^{r-2}\right)\cdot(\bu-\bv) \geq \tilde{C}\,\revone{\rF\|\bu-\bv\|^r_{0,r,\Omega}.}
     \end{align*}
     Finally, from the above inequality, taking \(\alpha^\frac12=\min\{1,\tilde{C}^\frac13\}\), it follows that 
     \begin{equation}\label{eq:strict} \langle \{A+C\}(\bu+\bw) - \{A+C\}(\bv + \bw), \bu - \bv\rangle \geq \alpha \| \bu - \bv \|_{\bV}^2, \end{equation}
where, we have used that the \(\bH(\vdiv)\)-seminorm is zero since \(\bu, \bv \in \mathbf{Z}\). 

\medskip 
\paragraph{\it Inf-sup stability \revone{and boundedness of $B$}} 
From the existence of the map $\cS_c$ defined in Lemma~\ref{lem:inf-sup}, it immediately follows that
\begin{align*}
\sup_{\bzero\neq\bv \in \bV} \!\!\frac{({\vdiv} \bv, p)}{\|\bv\|_\bV} &\geq \sup_{g \in \rQ'} \frac{({\vdiv} \cS_c(g), p)}{\|\cS_c(g)\|_\bV} = \sup_{g \in \rQ'} \frac{\langle g, p \rangle}{\|\cS_c(g)\|_\bV} \\&\geq \|\cS_c\|^{-1}_{\cL(\rQ', \bV)} \sup_{g \in \rQ'} \frac{\langle g, p \rangle}{\|g\|_{\rQ'}} = \beta \|p\|_{\rQ},
\end{align*}
where $\beta = \|\cS_c\|^{-1}_{\cL(Q', \bV)}$ is independent of the parameters $\kappa,\rF$, and hence, the inf-sup condition is satisfied. \revone{Note also that the boundedness of $B$ follows directly from H\'older's inequality  
\[ |b(\bv,q)| = |b(\kappa^{-\frac12}\bv,\kappa^{\frac12}q)| \leq \kappa^{-\frac12}\|\bv\|_{0,\Omega}\kappa^{\frac12}\|q\|_{0,\Omega} \leq \|\bv\|_{\bV}\|q\|_{\rQ}. \]
Therefore}, an appeal to the Theorem~\ref{th:abstract} implies the existence of a  unique velocity-pressure pair of solutions to the problem \eqref{eq:weak}. 

In order to complete the proof, it remains to show the continuous dependence  of 
$\|(\bu, p )\|_{\bV\times\rQ}$ on the data. From \eqref{eq:weak,1}, we arrive at \(B^*(p)=F-\{A+C\}(\bu)\), and the inf-sup stability of \(B\) shows
    \begin{align}\label{est:p_bound}
        \|p\|_{\rQ}\leq \frac{1}{\beta}\|B^*(p)\|_{\bV}\leq\revtwo{\|{F}\|_{\bV'}+\|\{A+C\}(\bu)\|_{\bV'}}.
    \end{align}
    Also, the surjectivity of \(B\) implies that since \(G\in\rQ' \text{ in } \eqref{eq:weak,2}\), there exists a unique \(\revone{\bu^m}\in \bV\setminus \bZ,\) such that \(\bu=\widetilde{\bu}+\revone{\bu^m},\) where \(\widetilde{\bu}\in \bZ,\) \[B(\revone{\bu^m})=G \quad \text{ and } \quad \|\revone{\bu^m}\|_{\bV}\leq \frac{1}{\beta}\|G\|_{\rQ'}.\]
    Now, for the {uniformly strong monotonicity} of \(A+C\) (cf. \eqref{eq:strict}), we choose  \(\bu=\widetilde{\bu},\bv=\bzero\) and \(\bw=-\revone{\bu^m}\), and then for equation \eqref{eq:weak,1}, we set \(\bv=\widetilde{\bu}\). Further, we use the local Lipschitz continuity  of \(A\), by setting   \(\bv=\bzero\) in \eqref{eq:lip},  to conclude that 
\begin{align*}
\alpha\|\widetilde{\bu}\|^{2}_{\bV}&\leq \langle \{A+C\}(\widetilde{\bu}+\revone{\bu^m})-\{A+C\}(\bzero+\revone{\bu^m}),\widetilde{\bu}\rangle \\& \leq (\|{F}\|_{\bV'}+\|\{A+C\}(\revone{\bu^m})\|_{\bV'})\|\widetilde{\bu}\|_{\bV}\\&
\leq (\|{F}\|_{\bV'}+\|\revone{\bu^m}\|_{\bV}+\|\revone{\bu^m}\|_{\bV}^{r-1})\|\widetilde{\bu}\|_{\bV} \\
&
\leq \bigl(\|{F}\|_{\bV'}+\frac{1}{\beta}\|G\|_{\rQ'}+\frac{1}{\beta^{r-1}}\|G\|_{\rQ'}^{r-1}\bigr)\|\widetilde{\bu}\|_{\bV},
\end{align*}
and hence,
using also the triangle inequality, we find that  
\begin{align}\label{eq:u_bound}
    \|\bu\|_{\bV}\leq \|\widetilde{\bu}\|_{\bV}+\|\revone{\bu^m}\|_{\bV}\leq \frac{1}{\alpha}\bigl(\|{F}\|_{\bV'}+\frac{1+\alpha}{\beta}\|G\|_{\rQ'}+\frac{1}{\beta^{r-1}}\|G\|_{\rQ'}^{r-1}\bigr). 
\end{align}
On the other hand, \revtwo{from 
\eqref{est:p_bound} and 
employing again the local Lipschitz continuity of \(A+C\) with \(\bv=\bzero\)  in \eqref{eq:lip}, we can invoke inequality  \eqref{eq:u_bound} to} arrive at 
    \begin{align*}
    \|p\|_{\rQ}& \leq\|{F}\|_{\bV'}+\|\bu\|_{\bV}+ \|\bu\|_{\bV}^{r-1}\\&\leq \|{F}\|_{\bV'}+\frac{1}{\alpha}\bigl(\|{F}\|_{\bV'}+\frac{1+\alpha}{\beta}\|G\|_{\rQ'}+\frac{1}{\beta^{r-1}}\|G\|_{\rQ'}^{r-1}\bigr)
   \\&\quad  +\frac{1}{\alpha^2}\bigl(\|{F}\|_{\bV'}+\frac{1+\alpha}{\beta}\|G\|_{\rQ'}+\frac{1}{\beta^{r-1}}\|G\|_{\rQ'}^{r-1}\bigr)^{r-1}\\
    &= \frac{1}{{\alpha}}\bigl((1+\alpha)\|{F}\|_{\bV'}+\frac{1+\alpha}{\beta}\|G\|_{\rQ'}+\frac{1}{\beta^{r-1}}\|G\|_{\rQ'}^{r-1}\bigr)\\
    &\quad    +\frac{1}{\alpha^2}\bigl(\|{F}\|_{\bV'}+\frac{1+\alpha}{\beta}\|G\|_{\rQ'}+\frac{1}{\beta^{r-1}}\|G\|_{\rQ'}^{r-1}\bigr)^{r-1}.
\end{align*}
This concludes the rest of the proof.
\end{proof}
\revtwo{Note that if $|\Gamma_{p}|=0$ then the weighted spaces take the form 
\[\bV=  \kappa^{-\frac{1}{2}}\bL^2(\Omega) \cap \bH_0(\vdiv,\Omega) \cap \revone{\rF^{\frac1r} \bL^r(\Omega)}, \quad \rQ = \rL_0^2(\Omega) + \kappa^{\frac{1}{2}}\widetilde{\rH}^1(\Omega) + \rF^{-\frac1r}\widetilde{\rW}^{1,r'}(\Omega),\]
where $\rL_0^2(\Omega) = \{q \in \rL^2(\Omega): \int_\Omega q = 0\}$, $\widetilde{\rH}^1(\Omega)=\{w\in \rH^1(\Omega): \int_\Omega w = 0\}$, and  $\widetilde{\rW}^{1,r'}(\Omega)=\{z\in \rW^{1,r'}(\Omega): \int_\Omega z = 0\}$. On the other hand, if $|\Gamma_{\bu}|=0$ then 
\[\bV=  \kappa^{-\frac{1}{2}}\bL^2(\Omega) \cap \bH(\vdiv,\Omega) \cap \revone{\rF^{\frac1r} \bL^r(\Omega)}, \quad \rQ = \rL^2(\Omega) + \kappa^{\frac{1}{2}}\rH_0^1(\Omega) + \rF^{-\frac1r}\rW_0^{1,r'}(\Omega).\]}
\section{Mixed finite element method and robust solvability}\label{sec:discrete}
Let \( {\cT}_h \) be a regular simplicial mesh defined in a bounded Lipschitz domain \( \Omega \), consisting of tetrahedral elements (or triangular elements in two dimensions) \( K \) with diameter \( h_K \). We define the mesh size as \( h := \max\{h_K : K \in {\cT}_h\} \). Given an integer \( k \geq 0 \) and a generic element \( K \in {\cT}_h \), we denote by \( \mathbb{P}_k(K) \) the space of polynomials defined locally in \( K \) with a degree at most \( k \), {and denote by \( \mathbb{P}_k(\cT_h) \) its global counterpart} 
\[
{\mathbb{P}_k(\cT_h)} := \left\{ p \in \rL^2(\Omega) : p|_K \in \mathbb{P}_k(K), \quad \forall K \in {\cT}_h \right\}.
\]
In addition, let ${\mathbb{RT}_k(\cT_h)}$ denote the \revone{discrete \( \mathbf{H}_{\Gamma_{\bu}}({\vdiv}, \Omega) \)-conforming} subspace defined by the Raviart--Thomas  elements of order \( k \), i.e.,  
\[
{\mathbb{RT}_k(\cT_h)} := \{ \bu \in \mathbf{H}_{\Gamma_{\bu}}({\vdiv}, \Omega) : \bu|_{K} = [\mathbb{P}_k(K)]^d \oplus \widetilde{\mathbb{P}}_k(K) \bx, \ \bx \in \Omega \subset \mathbb{R}^d, K \in {\cT}_h \},
\]
where we denote by \( \widetilde{\mathbb{P}}_k(K) \) the space of all polynomials over \( K \) of degree exactly equal to \( k \) and  \( '\oplus' \) denotes the direct sum of spaces.  With this we define the discrete weighted space for the velocity as \revone{all the vector fields in $\mathbb{RT}_k(\cT_h)$ such that their $\bV$ norm (cf. \eqref{eq:norms-V}) is bounded. Equivalently, and with the aim to emphasise the parameter-dependence, we write this space as} 
\begin{align}\label{eq:discrete.Vh}
 \bV\supseteq\bV_h:=  \kappa^{-\frac{1}{2}}\bL_h^2(\Omega) \cap \bH_{h,\Gamma_{\bu}}(\vdiv,\Omega) \cap \revone{\rF^{\frac1r} \bL_h^r(\Omega)},
 \end{align}
 where \( {\bL}_h^2(\Omega), {\bH}_{{h,\Gamma_{\bu}}}({\vdiv}, \Omega) , \) and \revone{\({\bL}_h^r(\Omega)\)} correspond to  functions in ${\mathbb{RT}_k(\cT_h)}$ equipped with the \( {\bL}^2(\Omega), {\bH}({\vdiv},\Omega),\) and \(\bL^r(\Omega)\)-norms, respectively.
 
In turn, define the discrete gradient operator \( \nabla_h: {\mathbb{P}_k(\cT_h)} \to {\mathbb{RT}_k(\cT_h)} \) by (\revone{cf.}, \cite[eqn. (3.1)]{arnold1997preconditioning})
\begin{align}\label{eq:disc.grad}
(\nabla_h q_h, \bv_h) = -(\vdiv\bv_h,q_h), \quad \forall q_h \in {\mathbb{P}_k(\cT_h)}, \text{ and } \forall \bv_h \in {\mathbb{RT}_k(\cT_h)}.
\end{align}
Additionally, define the following broken Sobolev spaces 
\begin{align*}
    {\rH^1}(\mathcal{T}_h)&:=\{p\in\rL^2(\Omega):p|_K\in\rH^1(K), \quad \forall K\in \mathcal{T}_h\},\\
    {\rW^{1,r'}}(\mathcal{T}_h)&:=\{p\in\rL^2(\Omega):p|_K\in\rW^{1,r'}(K), \quad \forall K\in \mathcal{T}_h\},
\end{align*}
and whenever needed we add a subscript $\Gamma_i$ denoting the part on the boundary where the value of the function is prescribed to zero. These spaces are equipped with DG (broken Sobolev) seminorms 
\(|\bullet|_{{1,\mathcal{T}_h}}\) and 
\(|\bullet|_{{1,r',\mathcal{T}_h}}\),
defined as 
\begin{align*}
{\rH^1}(\mathcal{T}_h) \ni \zeta & \mapsto  |\zeta|^2_{{1,\mathcal{T}_h}} := \biggl(
\sum_{K\in\mathcal{T}_h}\| {\nabla} \zeta\|^{2}_{0,K} 
+ \sum_{F\in \mathcal{F}_h} h_F^{-1}\|[\![\zeta]\!]\|^2_{0,F}
\biggr)^{\frac12}, 
\\
\rW^{1,r'}(\mathcal{T}_h) \ni \xi & \mapsto 
\revone{|\xi|_{{1,r',\mathcal{T}_h}} := \biggl(\sum_{K\in\mathcal{T}_h}\|{\nabla} \xi\|^{r'}_{0,r',K}
+ \sum_{F\in \mathcal{F}_h} h_F^{1-r'}\|[\![\xi]\!]\|^{r'}_{0,r',F}
\biggr)^{\frac{1}{r'}}},\end{align*}
respectively, where $\mathcal{F}_h$ denotes the set of facets in the mesh $\cT_h$. 

Furthermore, we define the following weighted space for pressure, which is a sum of the infinite-dimensional (but broken) spaces $\rL^2(\Omega)$, \(\rH^1_{\Gamma_p}(\mathcal{T}_h)\) and \revone{\(\rW^{1,r'}_{\Gamma_p}(\mathcal{T}_h)\)}: 
\begin{equation}\label{eq:broken.Qhat}
{\widehat{\rQ}}(\cT_h):=\rL^2(\Omega)+\kappa^\frac12\rH^1_{\Gamma_p}(\mathcal{T}_h)+\revone{\rF^{-\frac1r}\rW^{1,r'}_{\Gamma_p}(\mathcal{T}_h)}.
\end{equation}
In this space we define a broken weighted norm for any  \( q \in{\widehat{\rQ}(\cT_h)}\) as follows  
\begin{equation}\label{def:Qhat-norm}
\|q\|_{{\widehat{\rQ}(\cT_h)}}:= 
\inf_{\substack{
    q = q_{1} + q_{2} + q_{3} \\
    (q_{1},q_{2},q_{3}) \in \\ \rL^2(\Omega)
     \times\rH_{\Gamma_p}^1(\mathcal{T}_h)
    \times \revone{\rW_{\Gamma_p}^{1,r'}(\mathcal{T}_h)}
}}
\left(
    \|q_{1}\|_{0,\Omega} 
    +\kappa^{\frac12}\, |q_{2}|_{1,\mathcal{T}_h} 
    + \revone{\rF^{-\frac1r} |q_{3}|_{1,r',\mathcal{T}_h}}
\right).
\end{equation}
Finally, \revthree{we define the discrete pressure space as all the functions in $\mathbb{P}_k(\cT_h)$ whose $\widehat{\rQ}_h(\cT_h)$ norm is bounded (see Remark \ref{rem:Qh})}
\begin{align}\label{eq:disc.Ph}
\rQ_h :=  \kappa^{\frac12} \rH^1_{h,{\Gamma_p}}(\Omega)+\rL_h^2(\Omega) +\revone{\rF^{-\frac1r}\rW^{1,r'}_{h,\Gamma_p}(\Omega)},
\end{align}
where \( \rL_h^2(\Omega), \rH^1_h(\Omega), \) and \revone{\(\rW^{1,r'}_h(\Omega)\)} are polynomial spaces having discrete functions from  \( {\mathbb{P}_k(\cT_h)} \), and equipped with the \( \rL^2(\Omega)\) norm  and the broken \(\rH^1(\cT_h),\,\revone{\rW^{1,r'}(\cT_h)}\) seminorms, respectively. 

\begin{remark}\label{rem:Qh}We equip $\rQ_h$ with the $\|\bullet\|_{\widehat{\rQ}(\cT_h)}$ norm \revthree{since we cannot compute, for example, the continuous $\rH^1(\Omega)$ and $\rW^{1,r'}(\Omega)$ norms of functions in $\mathbb{P}_k(\cT_h)$. We do have} that $\rQ_h \subseteq \widehat{\rQ}(\cT_h)$, but $\rQ_h$ is not a subspace of $\rQ$. This entails a non-conforming (in the pressure) scheme, requiring to define a discrete bilinear form associated with the discrete divergence operator 
$b_h: \bV_h\times \rQ_h \to \mathbb{R}$ 
 as well as a discrete linear functional \(G_h:{\rQ_h}\to\mathbb{R}\), as follows
\[
\revtwo{b_h(\bv_h,q_h):=-\int_{\Omega}\vdiv\bv_h\,q_h, \quad  G_h(q_h):=-\int_{\Omega}g\,q_h,} \quad\forall \bv_h\in \bV_h,  q_h\in  {\rQ_h}.\]
\end{remark}
 Then, with the conforming velocity and non-conforming pressure FE spaces \eqref{eq:discrete.Vh} and \eqref{eq:disc.Ph}, 
 the mixed finite element formulation of \eqref{eq:strong} consists in  finding $(\bu_h,p_h)\in \bV_h\times \rQ_h$ such that 
\begin{subequations}\label{eq:weak_Galerkin}
    \begin{align}
    a(\bu_h,\bv_h) + c(\bu_h; \bu_h,\bv_h) +  {b_h(\bv_h,p_h)} & = F(\bv_h) \quad \forall \bv_h \in \bV_h,\label{eq:1st.Galerkin}\\
    {b_h(\bu_h,q_h)} & = {G_h(q_h)} \quad \forall q_h \in \rQ_h\label{eq:2nd.Galerkin}.
    \end{align}
\end{subequations}
Similarly to the continuous case, an analysis using $\bV_h = (\mathbb{RT}_k(\cT_h),\|\bullet\|_{\bH^{\revone{r}}(\vdiv,\Omega)})$ and $\rQ_h = (\mathbb{P}_k(\cT_h),\|\bullet\|_{\rL^2(\Omega)})$ can be performed (using, e.g.,  \cite[Theorem 4.1]{caucao2020conforming}, \cite[Theorem 3.5]{pan2012mixed}) however the coercivity on the kernel and inf-sup conditions (and therefore also the stability estimates) are not necessarily robust in \(\kappa,\rF\). To achieve such robustness we employ the weighted spaces  \eqref{eq:discrete.Vh} and \eqref{eq:disc.Ph}. 

\subsection{Robust discrete solvability}
First, we prove the discrete inf-sup condition, which is an essential ingredient of the proof of the parameter-robust unique solvability 
Theorem, discussed in this subsection.  
\begin{lemma}\label{eq:S_disc}
        There exists an operator \(\cS_d\) satisfying the following properties: 
        \begin{gather}\nonumber
    \cS_d \in \cL(\rQ_h', \bV_h),\quad 
    \|\cS_d\|_{\cL(\rQ_h', \bV_h)}\ \text{is independent of } \kappa,\rF, h\quad \text{and}\\
    ({\vdiv} [\cS_d (g_h)], \psi_h) = \langle g_h, \psi_h \rangle \quad \text{for all}\ g_h \in \rQ_h',\psi_h \in \rQ_h.
\end{gather}
    \end{lemma}
    \begin{proof}
        We first define a bounded linear operator \(\cS_d\) from three different spaces, and then, invoke the duality relations between sum and intersection between Banach spaces \eqref{eq:duality} to conclude \(\cS_d\in \cL(\rQ_h',\bV_h).\) 
        
   For the existence of \( {\cS_d} \in {\cL}\bigl((\rW_{h,\Gamma_p}^{1, q}(\Omega))', {\bL}^{\revone{q'}}_h(\Omega) \bigr) \)
with $q\in \{2,\revone{r'}\}$, let  
\( \phi_h \in \rW^{1,\revone{q'}}_{h,\Gamma_p}(\Omega)\)  for a given  \( g_h \in (\rW_{h,\Gamma_p}^{1,q}(\Omega))'\) with $\frac1q + \frac{1}{q'}=1$ be a unique solution of the following weak formulation
\begin{equation}
    \langle\nabla_h\phi_h,\nabla_h\psi_h\rangle=\langle g_h,\psi_h\rangle \qquad \forall \psi_h\in \rW^{1,q}_{h,\Gamma_p}(\Omega),\label{eq:aux03}
\end{equation}
owing to the Banach--Ne\v{c}as--Babu\v{s}ka theorem for elliptic problems in Banach spaces \cite{ern2021finite}. Then, it  suffices to define \( \cS_d (g_h) := \nabla_h \phi_h \), 
where $\nabla_h \phi_h \in \bL_h^{\revone{q'}}(\Omega)$.  The continuous dependence on data coming from \eqref{eq:aux03} also implies that $\|\cS_d (g_h)\|_{0,q';\Omega} \lesssim \|g_h\|$. This, in turn, shows
\(
\cS_d \in {\cL}\bigl( (\rW^{1,q}_{h,\Gamma_p}(\Omega))', {\bL}_h^{q'}(\Omega) \bigr)\), \revone{for $q=2$ and $q=r'$ corresponding to $q'=2$ and $q'=r$, respectively}.

Next,   the existence and boundedness of \(\cS_d\in\cL\bigl((\rL^2_h(\Omega))',\bH_{h,\Gamma_{\bu}}(\vdiv,\Omega)\bigr)\) 
is a direct consequence of the surjectivity of the divergence operator from $\bH_{\Gamma_{\bu}}(\vdiv,\Omega)$ onto $\rL^2(\Omega)$ and standard properties of the Raviart--Thomas interpolation (even in the case of mixed boundary conditions) \cite{boffi2013mixed,gatica2014simple}. 
    \end{proof}
Below, we prove the discrete inf-sup condition.
\begin{lemma}[Discrete inf-sup condition]\label{eq:discrte_inf-sup}
 There exists a positive constant $\beta_d$, independent of $h,\kappa$ and $\rF$ such that  
 \begin{align*}
    \|q_h\|_{{\widehat{\rQ}(\cT_h)}}
    \leq {\beta_d}\sup_{\bv_h\in\bV_h}\frac{|(q_h,\vdiv\bv_h)|}{\|\bv_h\|_{\bV}}.
\end{align*} 
\end{lemma}
\begin{proof}
As a consequence of Lemma \ref{eq:S_disc}, we can assert that  
    \begin{align*}
\sup_{\bzero\neq\bv_h \in \bV_h} \!\!\frac{|({\vdiv} \,\bv_h, p_h)|}{\|\bv_h\|_\bV} &\geq \sup_{g_h \in \rQ_h'} \frac{|({\vdiv} [\cS_d (g_h)], p_h)|}{\|\cS_d (g_h)\|_\bV} = \sup_{g_h \in \rQ_h'} \frac{|\langle g_h, p_h \rangle|}{\|\cS_d (g_h)\|_\bV}\\& \geq \|\cS_d\|^{-1}_{\cL(\rQ_h', \bV_h)} \sup_{g_h \in \rQ_h'} \frac{|\langle g_h, p_h \rangle|}{\|g_h\|_{{\rQ}'_h}} = \beta_d \|p_h\|_{\widehat{\rQ}(\cT_h)},
\end{align*}
which concludes the proof.
\end{proof}

\noindent 
The following result is on  robust well-posedness  of the discrete mixed problem \eqref{eq:weak_Galerkin}.
\begin{theorem}[Parameter-robust unique solvability]\label{th:discrete}
Consider the weighted spaces defined in \eqref{eq:discrete.Vh} and \eqref{eq:disc.Ph} with the corresponding weighted norms,
and assume  that $\bf\in \bL^2(\Omega),\, g\in \rL^2(\Omega)$. Then, there exists a unique  pair $(\bu_h,p_h)\in \bV_h\times\rQ_h$ of  solutions to the discrete problem \eqref{eq:weak_Galerkin}. Furthermore, there exists a positive constant $\hat{C}$ independent of \(h,\) and the model parameters $\kappa, \rF$, such that
\begin{equation}\label{est:disc.stab}
    \|(\bu_h, q_h) \|_{\bV\times {{\widehat{\rQ}(\cT_h)}}}  \leq \hat{C}\max\bigl\{\|\bf\|_{0,\Omega} + \| g \|_{0,\Omega} + \|g\|^{\revone{r-1}}_{0,\Omega},(\|\bf\|_{0,\Omega} + \| g \|_{0,\Omega} + \|g\|^{\revone{r-1}}_{0,\Omega})^{r-1}\bigr\}.
    \end{equation}
\end{theorem}
\begin{proof}
  The discrete well-posedness analysis uses the discrete version of Theorem~\ref{th:abstract}, detailed in \cite[Theorem 4.1]{caucao2020conforming}.  The local Lipschitz continuity of the discrete counterpart of \(\{A+C\}\) follows similarly as in Theorem \ref{th:continuous}, and the discrete inf-sup condition follows from Lemma~\ref{eq:discrte_inf-sup}. Hence,  to complete the rest of the proof, it is enough to prove the {uniformly strong monotonicity}.

\medskip
\noindent
{\it Uniformly strong monotonicity of \(\{a(\bullet,\bullet)+c(\bullet;\bullet,\bullet)\}(\bullet+\bw_h)\).} 
We start with the observation that since \(\vdiv \bV_h\subset \rQ_h,\) for all $\bu_h\in \bZ_h:=\{\bv_h \in \bV_h: b_h(\bv_h,q_h)=0 \quad \forall q_h\in \rQ_h\},$ we can choose $q_h=\vdiv \bu_h \in \rQ_h.$ Therefore, we characterise the discrete kernel as follows 
\[ \bZ_h= \{ \bv_h\in \bV_h: \vdiv\bv_h = 0\},\]
and we remark that \(\bZ_h \subseteq \bZ\), which directly follows from the characterisation of \(\bZ_h\). Now, for $\bu_h,\bv_h \in \bZ_h$  and ${\bw_h\in \bV_h}$, the following relation holds
     \begin{align*}
       &a(\bu_h+\bw_h,\bu_h-\bv_h)+c(\bu_h+\bw_h;\bu_h+\bw_h,\bu_h-\bv_h)-a(\bv_h+\bw_h,\bu_h-\bv_h)-\\&c(\bv_h+\bw_h;\bv_h+\bw_h,\bu_h-\bv_h)\\&
       =\int_\Omega \! \kappa^{-1}|\bu_h-\bv_h|^2\\
       &\quad 
         +\rF\int_\Omega  \left((\bu_h+\bw_h)|\bu_h+\bw_h|^{r-2}-(\bv_h+\bw_h)|\bv_h+\bw_h|^{r-2}\right)\cdot(\bu_h-\bv_h).
     \end{align*}
By employing \cite[Lemma 2.1, equation (2.1b)]{barrett1993finite} we deduce that there exists a constant \(\hat{C}>0\) independent of \(h,\kappa,\) and \(\rF,\) such that 
     \[\revone{
         \rF\int_\Omega\!  \!\!\left((\bu_h+\bw_h)|\bu_h+\bw_h|^{r-2}
         \!-\!(\bv_h+\bw_h)|\bv_h+\bw_h|^{r-2}\right)\!\cdot(\bu_h\! -\bv_h) \gtrsim \rF\|\bu_h \! -\bv_h\|^r_{0,r,\Omega}, 
}\]
and rearranging terms  we get
     \begin{align}\label{eq:monotone.disc} a(&\bu_h+\bw_h,\bu_h-\bv_h)+c(\bu_h+\bw_h;\bu_h+\bw_h,\bu_h-\bv_h)\nonumber\\&-a(\bv_h+\bw_h,\bu_h-\bv_h)-c(\bv_h+\bw_h;\bv_h+\bw_h,\bu_h-\bv_h) \geq \widetilde{\alpha} \| \bu_h - \bv_h \|_{\bV}^2,\end{align}
where we have used that the \(\bH(\vdiv)\)-seminorm is zero since \(\bu_h, \bv_h \in \mathbf{Z}_h\). Note also that \(\widetilde{\alpha}\) is independent of \(h,\kappa\), and \(\rF\). 

Thus, an appeal to the discrete version of  the Theorem ~\ref{th:abstract} ensures  existence and  uniqueness of the discrete solution. Finally, the discrete stability estimate~\eqref{est:disc.stab} follows similarly to the continuous case (Theorem \ref{th:continuous}), and we observe that the  discrete stability constant  is independent of \(h\), \(\kappa\), and \(\rF\). 
\end{proof}

\subsection{ {\it{A priori}} error analysis}\label{sec:apriori}
\revone{Now we prove} a  theorem on {\it a priori} error estimates for the Darcy--Forchheimer model, which are robust with respect to \(\kappa\), \(h\), and \(\rF\). \revone{We recall, for a given subspace $S_h$ of the generic Banach space $(S,\|\bullet\|_S)$, the notation $\mathrm{dist}(v,S_h):=\inf_{w\in S_h}\|v-w\|_{S}$.}
\begin{theorem}[Quasi-optimality]\label{th:cea}
Let $(\bu, p)$ uniquely solve the continuous model (\ref{eq:weak}), and let the pair 
\((\bu_h, p_h)\) be the solution of the discrete model (\ref{eq:weak_Galerkin}). Then, the following quasi-optimal error estimates for the velocity and pressure hold
\begin{align}\label{est:Cea}
\|\bu - \bu_h\|_{\bV}+\|p - p_h\|_{{\widehat{\rQ}(\cT_h)}} \leq {C_c} \bigl( \mathrm{dist}( \bu, \bV_h ) +& \mathrm{dist}(p, {\rQ}_h)\bigr), 
\end{align}
where \({C_c}\) is independent of $h,\kappa$ and $\rF$.
\end{theorem}
\begin{proof} Now, define  
\begin{align}\label{eqn:Zhg}
  \bZ_h^g:=\{\bv_h\in\bV_h:b_h(\bv_h,q_h)=G(q_h)\quad \forall q_h\in {\rQ}_h\}.  
\end{align} 
Note that $\bu_h\in\bZ_h^g$ and $\bu_h-\bu_h^g\in\bZ_h$ for all $\bu_h^g\in \bZ_h^g$.
From triangle inequality, an application of uniformly strong monotonicity of \({a(\bullet,\bullet)+c(\bullet;\bullet,\bullet)}\) to the pair $(\bu_h-\bu_h^g,\bzero)\in \bZ_h\times\bZ_h$, and choosing $\bw_h = \bu_h^g$ in \eqref{eq:monotone.disc}, we arrive at 
\[\widetilde{\alpha}\|\bu_h-\bu_h^g\|^2_{\bV}\leq {a(\bu_h,\bu_h-\bu_h^g)+c(\bu_h;\bu_h,\bu_h)-a(\bu_h^g,\bu_h-\bu_h^g)-c(\bu_h^g,\bu_h-\bu_h^g)}.\]
Further, adding and subtracting \({a(\bu,\bu_h-\bu_h^g)+c(\bu;\bu,\bu_h-\bu_h^g)}\) to the right-hand side of the above inequality, the following bound is obtained
\begin{align}\label{eqn:uh-uhg}
\widetilde{\alpha}\|\bu_h-\bu_h^g\|^2_{\bV}& \leq a(\bu,\bu_h-\bu_h^g)+c(\bu;\bu,\bu_h-\bu_h^g)-a(\bu_h^g,\bu_h-\bu_h^g)\nonumber\\
&\quad -c(\bu_h^g;\bu_h^g,\bu_h-\bu_h^g)+a(\bu_h,\bu_h-\bu_h^g)+c(\bu_h;\bu_h,\bu_h-\bu_h^g)\nonumber\\&\quad -a(\bu,\bu_h-\bu_h^g)-c(\bu;\bu,\bu_h-\bu_h^g).
\end{align}
Next, from \eqref{eq:1st.Galerkin}, we observe that  \begin{align}\label{eqn:A=B}
a(\bu,\bu_h-\bu_h^g)+c(\bu;\bu,\bu_h-\bu_h^g)-a(\bu_h,\bu_h-\bu_h^g)\nonumber&-c(\bu_h;\bu_h,\bu_h-\bu_h^g)\\& =b_h(\bu_h-\bu_h^g,p-p_h).
\end{align}
Then, combining \eqref{eqn:uh-uhg}, \eqref{eqn:A=B}, the local Lipschitz continuity of \({a(\bullet,\bullet)+c(\bullet;\bullet,\bullet)}\), and the boundedness of \({b_h(\bullet,\bullet)}\), we arrive at 
\begin{align*}
    \widetilde{\alpha}\|\bu_h-\bu_h^g\|_{\bV}\leq  \|p-p_h\|_{{\widehat{\rQ}(\cT_h)}}+ (1+\|\bu\|_{\bV}+\|\bu_h^g\|_{\bV})\|\bu-\bu_h^g\|_{\bV}.
\end{align*}
Finally, a use of the triangle inequality yields the estimate for the velocity error as
\begin{align}
\|\bu-\bu_h\|_{\bV}&\leq\frac1{\widetilde{\alpha}}\|p-p_h\|_{{\widehat{\rQ}(\cT_h)}}+ \frac1{\widetilde{\alpha}}((1+\widetilde{\alpha})+\|\bu\|_{\bV}+\|\bu_h^g\|_{\bV})\|\bu-\bu_h^g\|_{\bV}\nonumber\\
&\leq C_g(\mathrm{dist}(p,{\rQ}_h)+\mathrm{dist}(\bu,\bV_h)),
\end{align}
where \(C_g := 
\revone{\frac1{\widetilde{\alpha}}(1+\frac1{\beta_d})
((1+\widetilde{\alpha})+\|\bu\|_{\bV}+\|\bu_h^g\|_{\bV})}
\), which is obtained by using that \(\mathrm{dist}(\bu,\bZ_h^g)\leq (1+\frac1{\beta_d})\,\mathrm{dist}(\bu,\bV_h)\) (see for instance, \cite[Theorem 2.6]{gatica2014simple}).

Now, for the pressure estimate, we use the inf-sup stability of \(b_h(\bullet,\bullet).\) Let \(p_h^*\in {\rQ}_h\)
 be an arbitrary element. Then, there holds 
 \begin{align*}
   & \beta_d\|p_h-p_h^*\|_{{{\widehat{\rQ}}(\cT_h)}}\leq\sup_{\bzero\neq\bv_h\in\bV_h}\frac {b_h(\bv_h,p_h-p_h^*)}{\|\bv_h\|_{\bV}}\\
    &\quad =\sup_{\bzero\neq\bv_h\in\bV_h}\frac{b_h(\bv_h,p-p_h^*)+b_h(\bv_h,p_h-p)}{\|\bv_h\|_{\bV}}\\
    &\quad =\sup_{\bzero\neq\bv_h\in\bV_h}\frac{b_h(\bv_h,p-p_h^*)+\revone{a(\bu_h-\bu,\bv_h) + c(\bu_h;\bu_h,\bv_h) - c(\bu;\bu,\bv_h)}}{\|\bv_h\|_{\bV}}\\
    &\quad \leq\|p-p_h^*\|_{{\widehat{\rQ}(\cT_h)}}+(1+\|\bu\|_{\bV}+\|\bu_h^g\|_{\bV})\|\bu-\bu_h\|_{\bV},
\end{align*}
\revone{where we have used the boundedness of $b_h(\bullet,\bullet)$ and $a(\bullet,\bullet)$, the local Lipschitz continuity of the map $\bu \mapsto |\bu|^{r-2}\bu$, and the boundedness of the discrete solution. Then we choose $p_h^*$ as the best approximation, use triangle inequality, and employ the velocity estimate to obtain}
\begin{align}
    \|p-p_h\|_{{\widehat{{\rQ}}(\cT_h)}}\leq \widetilde{C}_g(\mathrm{dist}(p,{{\rQ}_h})+\mathrm{dist}(\bu,\bV_h)),
\end{align}
where \(\widetilde{C}_g:=
\revone{1+\frac1{\beta_d}+C_g^{\frac12}(1+\|\bu\|_{\bV}+\|\bu_h^g\|_{\bV})}
\). This concludes the proof.
\end{proof}
Additionally, if  \(\bZ_h\subseteq \bZ,\) then, we obtain the following pressure-independent estimate (cf. \cite[Theorem B.2]{gatica2014simple}) for the velocity.
\begin{corollary}\label{cor:45}
Under the assumptions of Theorem~\ref{th:cea}, there holds for any \(\bu_h^g\in\bZ_h^g\)
\begin{equation}\label{est:Cea_pressure.independent}
    \|\bu-\bu_h\|_{\bV}\leq C_D \,\mathrm{dist}(\bu,\bV_h),
\end{equation}
where \(\revone{0<C_D:=\frac1{\widetilde{\alpha}}((1+\widetilde{\alpha})+\|\bu\|_{\bV}+\|\bu_h^g\|_{\bV})(1+\frac1{\beta_d})<\infty}\) is 
independent of \(h,\kappa,\) and \(\rF\). 
\end{corollary}
\begin{proof}
From the characterisation of \(\bZ_h\), there holds $b_h(\bv_h,q-q_h)=0$ for all $\bv_h\in \bZ_h$. 
Then, from \eqref{eq:2nd.Galerkin}, it is evident that \(\bu_h\in\bZ_h^g\). Now, let \(\bu_h^g\) be an arbitrary, but fixed element in \(\bZ_h^g\), then   \(\bu_h-\bu_h^g\in \bZ_h\). A use of the triangle inequality shows
\[\|\bu-\bu_h\|_{\bV}\leq\|\bu-\bu_h^g\|_{\bV}+\|\bu_h-\bu_h^g\|_{\bV}.
\]
Apply the uniformly strong monotonicity of $a(\bullet,\bullet)+c(\bullet;\bullet,\bullet)$ to the pair $(\bu_h-\bu_h^g,\bzero)\in \bZ_h\times\bZ_h$ and also take $\bw_h = \bu_h^g$ in \eqref{eq:monotone.disc} to arrive at 
\begin{align}\label{Vhg}
    &\widetilde{\alpha}\|\bu_h-\bu_h^g\|^2_{\bV}\\
  \nonumber  &\revone{\leq} 
    a(\bu_h,\bu_h-\bu_h^g)+c(\bu_h;\bu_h,\bu_h-\bu_h^g)-a(\bu_h^g,\bu_h-\bu_h^g)-c(\bu_h^g;\bu_h^g,\bu_h-\bu_h^g).
\end{align}
Observe that, by adding and subtracting \(a(\bu,\bu_h-\bu_h^g)+c(\bu;\bu,\bu_h-\bu_h^g)\) on the right-hand side of \eqref{Vhg}, and using \eqref{eq:weak,1} and \eqref{eq:1st.Galerkin}, as well as the fact that \(b_h(\bu_h-\bu_h^g,q_h)=0\) for all \(q_h\in {\rQ}_h,\) we deduce the following:
\begin{align*}
     &\widetilde{\alpha}\|\bu_h-\bu_h^g\|_{\bV}\\
     &\leq b_h(\bu_h,\bu_h^g,p-p_h)+a(\bu,\bu_h-\bu_h^g)+c(\bu;\bu,\bu_h-\bu_h^g)-a(\bu_h^g,\bu_h-\bu_h^g)\\&\qquad -c(\bu_h^g;\bu_h^g,\bu_h-\bu_h^g)\\
     &=a(\bu,\bu_h-\bu_h^g)+c(\bu;\bu,\bu_h-\bu_h^g)-a(\bu_h^g,\bu_h-\bu_h^g)-c(\bu_h^g;\bu_h^g,\bu_h-\bu_h^g)\\
     &\leq (1+\|\bu\|_{\bV}+\|\bu_h^g\|_{\bV})\|\bu-\bu_h^g\|_{\bV},
\end{align*}
where we have used that \(\bZ_h\subseteq\bZ\) in the second step, and the local Lipschitz property of \(a(\bullet,\bullet)+c(\bullet;\bullet,\bullet)\) in the third step. 
Finally, a use of the triangle inequality yields 
\[\|\bu-\bu_h\|_{\bV}\leq\frac1{\widetilde{\alpha}}((1+\widetilde{\alpha})+\|\bu\|_{\bV}+\|\bu_h^g\|_{\bV})\mathrm{dist}(\bu,\bZ_h^g).\]
Recalling from \cite[Theorem 2.6]{gatica2014simple} that  \(\mathrm{dist}(\bu,\bZ_h^g)\leq (1+\frac1{\beta_d})\mathrm{dist}(\bu,\bV_h)\),where \(\beta_d\) is the discrete inf-sup constant, we achieve the desired pressure-robustness.
\end{proof}

Let us remark that by just considering the spaces \(\mathbf{H}^r_{\Gamma_{\bu}}({\vdiv},\Omega)\) for the velocity and \(\rL^2(\Omega)\) for the pressure and following the approach of \cite{pan2012mixed} one can obtain the pressure and velocity  error estimates separately, but they are not necessarily robust with respect to the parameters \(\kappa, \rF\).

In order to provide a theoretical rate of convergence for the mixed FEM \eqref{eq:weak_Galerkin}, we recall the approximation properties of the subspaces involved. 

\medskip 
\revtwo{Let $t \in (1,\infty)$ and denote by  \(\Pi_h:\bW^{l,t}(\Omega)\cap\bH(\vdiv_{1,t},\Omega)\to\mathbb{RT}_k(\cT_h)\)  the Fortin operator and by \(P_h:\rW^{l,t}(K)\to \mathbb{P}_k(K)\),  the local orthogonal \(\rL^2\) projection. Then, as a standard consequence of the Bramble--Hilbert lemma, optimal approximation, and inverse inequalities on polynomial spaces, we can assert that} 
\begin{subequations}
     \begin{align}
\label{eq:approx-v}         &  \revtwo{\|\bv- \Pi_h \bv\|_{m,t,\Omega}\leq C_1 h^{l-m}|\bv|_{l,t,\Omega}\quad \forall \bv\in \mathbf{W}^{l,t}(\Omega), \ 1\leq l \leq k+1,}\\
  \label{eq:approx-div}  & \revtwo{ \|\vdiv(\bv-\Pi_h\bv)\|_{m,t,\Omega}\leq C_2h^{l-m}|\vdiv\bv|_{l,t,\Omega} \quad \forall \bv \in \mathbf{W}^{1,t}(\Omega),\ \text{with}}\\
    \nonumber           & \qquad \qquad \qquad \qquad \qquad \qquad \qquad \qquad \qquad  \revone{\vdiv(\bv)\in \rW^{1,t}(\Omega)\ \text{and} \ 0\leq l \leq k+1,}\\
  \label{eq:approx-p}  & \revtwo{\|q-P_hq\|_{m,t,\Omega}\leq C_2 h^{l-m}|p|_{l,t,\Omega}\quad \forall q\in \rW^{l,t}(\Omega),\ 0\leq l \leq k+1,} \\
    \label{eq:approx-p-K}  & \revtwo{\|\nabla(q-P_hq)\|_{0,t,{K}}\leq C_3 h_K^{l}|p|_{l+1,t,K}\quad \forall q\in \rW^{l+1,t}({K}),\  K\in\mathcal{T}_h,\ l \leq k,} 
     \end{align}
     \end{subequations}
     for $ 0\leq m\leq l$, where \revone{$C_1,C_2>0$ are  independent of \(h\) and $C_3>0$ is independent of $h_K$ 
(see \cite[Sections 4.1 and 4.5]{gatica2022p} for the first three, and \cite{ern2021finite} for the fourth  one)}.  

To conclude this section, the following theorem provides the theoretical rate of convergence 
for 
\eqref{eq:weak}, under suitable regularity assumptions on the exact solution. 
   \begin{theorem}[Convergence rates]\label{th:convergence}
        Let $(\bu,p) \in \bV \times \rQ$ be the solution of the continuous problem \eqref{eq:weak}  with the following \revone{additional} regularity conditions:
\begin{gather*}\bu\in \kappa^{-\frac12}\revone{\bH^{l}}(\Omega)\cap\revone{\rF^{\frac1r}\mathbf{W}^{l,r}(\Omega)},  \quad \vdiv\bu\in \revone{\rH^l}(\Omega), \quad \text{and} \quad \revtwo{p=p_1+p_2+p_3}, \\
\text{with} \quad \revtwo{p_1\in \rH^{l}(\Omega), \quad p_2\in \kappa^{\frac12}\rH^{l+1}(\Omega), \quad \text{and}\quad  p_3 \in \rF^{-\frac1r}\rW^{l+1,r'}(\Omega)},\end{gather*} where \revone{\(1\leq l\leq k+1\).  
Then, there exist  positive constants \(C_\sharp,C_\flat>0\)}, independent of \(h, \kappa\), and \(\rF\), such that 
\begin{align*}
  \|\bu-\bu_h\|_{\bV}+\|p-p_h\|_{\widehat{\rQ}(\cT_h)}&\leq 
  \revone{C_\sharp h^{l}\bigl[\kappa^{-\frac12}\|\bu\|_{l,\Omega}+\|\vdiv\bu\|_{l,\Omega}+\rF^{\frac1r}\|\bu\|_{l,r,\Omega}\bigr]}\\
  &\quad + \revone{C_\flat h^{l}\bigl[ |p_1|_{l,\Omega}+\kappa^{\frac12}|p_2|_{l+1,\Omega}+\rF^{-\frac1r}|p_3|_{l+1,r',\Omega}\bigr]}.
\end{align*}
    \end{theorem}
    \begin{proof}
    \revtwo{
    We first take $m=0$ in \eqref{eq:approx-v}, leading to  
\[ \| \bu - \Pi_h \bu\|_{0,t,\Omega} \leq C_1 h^l |\bu|_{l,t,\Omega}, \]
    and we use this bound with $t\in \{2,r\}$. We also take $m=0$ and $t=2$ in \eqref{eq:approx-div},   which gives 
    \[ \|\vdiv (\bu - \Pi_h \bu)\|_{0,\Omega} \leq C_2 h^l |\bu|_{l,\Omega},\]
and select $ 1\leq l \leq k+1 $ for the last two displayed estimates. Therefore, owing to  Corollary~\ref{cor:45}, we arrive at the velocity error bound  
\[ \|\bu - \bu_h\|_{\bV}  \leq C_D \mathrm{dist}(\bu,\bV_h) \leq C_DC_1^2C_2 h^l (\kappa^{-\frac12}\|\bu\|_{l,\Omega} + \|\vdiv\bu\|_{l,\Omega} + \rF^{\frac1r}\|\bu\|_{l,r,\Omega}).\]
For the pressure we employ the  decomposition $p=p_1+p_2+p_3$, the properties of the sum norm \eqref{def:Qhat-norm}, apply  \eqref{eq:approx-p} with $m=0$ and  $t=2$, and \eqref{eq:approx-p-K} once with   $t=2$, and once with   $t=r'$, both summed over $K\in \cT_h$. We combine that with the following trace inequality to control the jump term in the broken norms
\[h_F^{1-t} \|[\![p_j - P_h p_j]\!] \|^t_{0,t,F} =h_F^{1-t} \| [\![P_h p_j]\!]\|^t_{0,t,F} \leq C_4 h^{lt}|p_j|^t_{l+1,r,\omega_F}, \]
where $\omega_F$ stands for the usual face patch.  We apply this for $j=2,3$ with $t=2$ and $t=r'$), and we readily obtain   
\begin{align*} 
\|p - P_h p\|_{\widehat{\rQ}(\cT_h)} & \leq \| p_1 - P_hp_1\|_{0,\Omega} + \kappa^{\frac12}|p_2-P_hp_2|_{1,\cT_h} + \rF^{-\frac1r}|p_3-P_hp_3|_{1,r',\cT_h} \\
&\leq \max\{C_2,C_3,C_4\} h^{l}\bigl(|p_1|_{l,\Omega} +  \kappa^{\frac12}|p_2|_{l+1,\Omega} +  \rF^{-\frac1r}|p_3|_{l+1,r',\Omega} \bigr).
\end{align*}
Then we can invoke directly the quasi-optimality estimate \eqref{est:Cea}, to get the desired result with 
$C_\sharp = C_D(2C_1+C_2)$ and $ C_{\flat} = \widetilde{C}_g\max\{C_2,C_3,C_4\}$.}
Since the last two contributions of the pressure norm  require derivatives, the $\rL^2$ projection only delivers theoretical approximability if $k\geq 1$ or if 
  the infimum decomposition can be chosen such that $p = p_1$ (and $p_2 = p_3 =0$), 
i.e., if the optimal decomposition lives entirely in the $\rL^2$ component. 
\end{proof}
    
\section{Preconditioning for the linearised Darcy--Forchheimer equations}\label{sec:precond}
In this section, we define a variable operator preconditioner for the Newton linearisation of  \eqref{eq:weak}. Using the notation from  Section 3, we first rewrite such system in operator form as
\begin{equation}\label{eq:operator-form}
    \begin{pmatrix}
        A + C & B^* \\ B & 0 \end{pmatrix}\begin{pmatrix}
            \bu \\ p 
        \end{pmatrix} = \begin{pmatrix}
            F \\ G
        \end{pmatrix} \qquad \text{in $\bV'\times \rQ'$}.
\end{equation}

Let us denote by \(\mathcal{H}(\bu)\) the G\^ateaux derivative of the nonlinear operator $C$, i.e., \(|\bu|^{r-2}\bu\) at the state \(\bu\) in the direction of \(\delta\bu\). In addition, consider now the Newton--Raphson linearisation of \eqref{eq:operator-form}, which starting from an initial guess $\bu^0,p^0$, and iterating on  $m = 1, \ldots$ until convergence, seeks velocity and pressure increments $(\delta \bu, \delta p)$ in a new space 
\begin{align}\label{eq:linearised-spaces}\underline{\bV}\times \underline{\rQ} &:=\bigl[\kappa^{-\frac12}\bL^2(\Omega)\cap\bH_{\Gamma_{\bu}}(\vdiv,\Omega)\cap\revone{[\rF\mathcal{H}(\revone{\bu^m})]^\frac1r\bL^r(\Omega)}\bigr]\nonumber\\ &\quad \times \bigl[\rL^2(\Omega)+\kappa^\frac12\rH_{\Gamma_p}^1(\Omega)+\revone{[\rF\mathcal{H}(\revone{\bu^m})]^{-\frac1r}\rW_{\Gamma_p}^{1,r'}(\Omega)}\bigr]
,\end{align} 
such that 
 \begin{equation}\label{eq:newton} 
\begin{pmatrix}\kappa^{-1}\mathcal{I}+\rF 
\mathcal{H}({\bu}^m)\mathcal{I} & B^* \\
    B  & 0\end{pmatrix}\begin{pmatrix}\delta \bu \\ \delta p\end{pmatrix}= \begin{pmatrix}
      F - \{A+C\}({\bu}^m) - B^*({p}^m) \\ G - B({\bu}^m)  
    \end{pmatrix} \quad \text{in $\underline{\bV}'\times \underline{\rQ}'$}. \end{equation}
Here, $\mathcal{I}$ denotes the identity operator  at the continuous level, but for the discrete case it stands for the corresponding mass matrix (either for velocity or pressure). Further, we note that, in weak form, the action of the operator $\rF\mathcal{H}({\bu}^m)\mathcal{I}:\underline{\bV}\to \underline{\bV}'$ reads
\[ \langle \rF \mathcal{H}({\bu}^m)(\delta \bu), \bv \rangle = \int_\Omega \rF |\bu^m|^{r-2}\delta\bu \cdot \bv + \int_\Omega \rF(r-2) |\bu^m|^{r-4}(\bu^m \cdot \delta\bu) (\bu^m \cdot \bv). \]
Then, update the solution as \((\bu^{m+1},p^{m+1}) = (\bu^m + \delta \bu, p^m + \delta p)\). 
Next Lemma is on the solvability of the system \eqref{eq:newton}.
\begin{lemma}\label{lem:solv-jac}
 Assume that  for any $\bw \in \underline{\bV},$ the function $\rF\mathcal{H}(\bw)$ is in {$\bL^{\infty}(\Omega)$} and that $\rF\mathcal{H}(\bw)(\bx) \geq C > 0$ a.e. in $\Omega$. Then, the problem \eqref{eq:newton} has a unique solution in $\underline{\bV} \times \underline{\rQ}$.  
\end{lemma}
\begin{proof}
We use the classical Babu\v{s}ka--Brezzi theory for saddle-point problems. 
The bilinear form $\tilde{a}:\underline{\bV}\times\underline{\bV}\to\mathbb{R}$ defined as 
    \[ \tilde{a}(\bu,\bv) := a(\bu,\bv) + \langle \rF \mathcal{H}({\bu}^m)(\bu), \bv \rangle,\]
    is bounded in $\underline{\bV}$ using the H\"older's inequality and coercive in $\underline{\bZ} = \mathrm{ker} B$; all with constants independent of $\kappa$, $\rF$, and $\mathcal{H}(\bu^m).$ Finally, the boundedness and inf-sup condition on $B$ have been shown in the proof of Theorem~\ref{th:continuous} for a different weight. In this case, we only need to alter the scaling, \revtwo{thus completing} the proof.
\end{proof}
\revtwo{Under the assumption in Lemma~\ref{lem:solv-jac}, the weighted $r-$Laplacian defines a coercive and bounded operator on $\rW^{1,r}_{\Gamma_p}(\Omega)$ with constants independent of the discretisation parameters. 
From a practical standpoint, this corresponds to freezing the nonlinear coefficient at the current iterate and using a uniformly elliptic linear operator at each step. In all numerical experiments reported in Section~\ref{sec:numerics}, the iterates remain in a regime where these bounds are satisfied, and no degeneracy or loss of stability is observed.}

The spaces defined in \eqref{eq:linearised-spaces} induce a duality map that can be used for preconditioning the system \eqref{eq:newton}, which is discussed below.
\begin{theorem}[Local preconditioner for the tangent system]\label{lem:precond_Mardal}
Let  
   \((\revone{\bu^m},\revone{p^m})\in\underline{\bV}\times \underline{\rQ}\) be the state around which the Newton--Raphson linearisation is performed. 
    Then the following block diagonal operator  \begin{align}\label{precond.linear}\begin{pmatrix}
        \mathcal{P}:=(\kappa^{-1}\mathcal{I}+\rF 
        {\mathcal{H}(\revone{\bu^m})} {\mathcal{I}}-\nabla\vdiv)^{-1} & 0\\0 &\!\!\!\! \!\!\!\!\!\!\!\mathcal{Q}:=\mathcal{I}^{-1}+ {\bigl(-\kappa\Delta\bigr)^{-1}-\bigl([\rF{\mathcal{H}(\revone{\bu^m})}]^{r-4}{\Delta_{\revone{r'}}}\bigr)^{-1}}
    \end{pmatrix},\end{align} 
   is a parameter-robust preconditioner for  \eqref{eq:newton}, in the sense that the condition numbers of the preconditioned matrix are uniformly bounded in \(\kappa\) and \(\rF\). Here, 
making abuse of notation, by $[\rF{\mathcal{H}(\revone{\bu^m})}]^{r-4}{\Delta_{r'}}$ we denote the (singular) weighted p-Laplace operator associated with the space \revone{$\rW^{1,r'}_{\Gamma_p}(\Omega)$}, defined as  
  \[\revone{\langle [\rF{\mathcal{H}(\revone{\bu^m})}]^{r-4}\Delta_{r'}\phi,\psi\rangle := \int_\Omega [\rF{\mathcal{H}(\revone{\bu^m})}]^{r-4}|\nabla\phi|^{r'-2}\nabla \phi \cdot \nabla \psi, \qquad \forall \psi \in \rW^{1,r}(\Omega).} \] 
\end{theorem}
\begin{proof}
Preconditioners for the upper diagonal block in system \eqref{eq:newton} can be directly constructed following the approach of \cite[Equation 7]{baerland2020observation}. This straightforwardly leads to $\mathcal{P}$ in \eqref{precond.linear}. We now focus on showing that $\mathcal{Q}$ is a preconditioner for the pressure block. Observe that the norm for the space \(\underline{\rQ}\) in \eqref{eq:linearised-spaces} can be written as  %
    \begin{equation}\label{norm:Q_equivalent}
        \|q\|^2_{\underline{\rQ}}:=\inf_{\varphi\in\rH^1_{\Gamma_p}(\Omega)}\bigl\{\|q-\varphi\|^2_{0,\Omega}+\kappa\|\nabla\varphi\|^2_{0,\Omega}+\revone{{[\rF{\mathcal{H}(\revone{\bu^m})}]^{r-4}}\|{\nabla}\varphi\|^{r'}_{0,r',\Omega}}\bigr\}.
    \end{equation}
%
Note that the above infimum is attained by a \(\varphi\in \rH^1_{\Gamma_p}(\Omega)\) (therefore also in $\rL^2(\Omega)$ and $\rW^{1,r'}_{\Gamma_p}(\Omega)$) that satisfies  
\begin{subequations}\label{elliptic}
\begin{align}
       (\mathcal{I}-\kappa\Delta-[\rF{\mathcal{H}(\revone{\bu^m})}]^{r-4}{{\Delta}_{r'}})\varphi&=q \quad \text{in } \Omega , \\
       (\kappa \nabla \varphi + [\rF{\mathcal{H}(\revone{\bu^m})}]^{r-4}|\nabla \varphi|^{r'-2}\nabla\varphi)\cdot\bn &=0 \quad \text{on } \Gamma_{\bu},\label{elliptic 2}\\
       \varphi &= 0 \quad \text{on } \Gamma_p. \label{elliptic 3}
   \end{align}
   \end{subequations}
   
The unique solvability of \eqref{elliptic} follows from the Minty--Browder framework for monotone operators, stressing that \(q\) is from a suitable dual space (following from Sobolev embedding). 
Using \eqref{elliptic}, we obtain the following 
   \begin{align*}
   \|q\|^2_{\underline{\rQ}}&=\|q-\varphi\|^2_{0,\Omega}+\kappa\|\nabla\varphi\|^2_{0,\Omega}+[\rF{\mathcal{H}(\revone{\bu^m})}]^{r-4}\|{\nabla}\varphi\|^{r'}_{0,r',\Omega}\\
   &= \langle-\kappa\Delta\varphi-[\rF{\mathcal{H}(\revone{\bu^m})}]^{r-4}{\Delta}_{r'}\varphi,-\kappa\Delta\varphi-[\rF{\mathcal{H}(\revone{\bu^m})}]^{r-4}{\Delta}_{r'}\varphi\rangle -\kappa\langle\Delta\varphi,\varphi\rangle\\&\qquad -\langle [\rF{\mathcal{H}(\revone{\bu^m})}]^{r-4}{\Delta}_{r'}\varphi,\varphi\rangle\\
   &=\langle(-\kappa\Delta-[\rF{\mathcal{H}(\revone{\bu^m})}]^{r-4}{\Delta}_{r'})\varphi,(\mathcal{I}-\kappa\Delta-[\rF{\mathcal{H}(\revone{\bu^m})}]^{r-4}{\Delta}_{r'})\varphi\rangle.
   \end{align*} 
   Next,    we employ again \eqref{elliptic} 
   to conclude that \begin{equation}\label{can.precon}\|q\|^2_{\underline{\rQ}}=\revone{\langle(-\kappa\Delta-[\rF{\mathcal{H}(\revone{\bu^m})}]^{r-4}{\Delta}_{r'})(\mathcal{I}-\kappa\Delta-[\rF{\mathcal{H}(\revone{\bu^m})}]^{r-4}{\Delta}_{r'})^{-1}q,q\rangle}.\end{equation}
   Then, using the duality map we readily define the canonical preconditioner \(\mathcal{Q}:\underline{\rQ}'\to \underline{\rQ}\) associated with  \eqref{can.precon}  by \[\revone{\mathcal{Q}=(\mathcal{I}-\kappa\Delta-[\rF{\mathcal{H}(\revone{\bu^m})}]^{r-4}{\Delta}_{r'})^{-1}}.\] 
  Therefore,  proceeding similarly as in \cite[Section 3]{mardal2011preconditioning},  we can assert that \eqref{precond.linear} is a  \((\kappa,\rF)\)-robust block diagonal preconditioner for  system \eqref{eq:newton}.
   This concludes the proof.
\end{proof}
\revtwo{The realisation of $\mathcal{Q}$ in  \eqref{precond.linear} requires solving a nonlinear $r'$-Laplacian problem, which would indeed be computationally expensive. In practice, we do not apply the nonlinear inverse. Instead, the nonlinear coefficient is frozen at the current iterate, yielding a linear, weighted elliptic operator. The application of
$\mathcal{Q}^{-1}$   is therefore reduced to the solution of a linear problem with variable coefficients.
In the present implementation, this linear system is solved using a sparse direct solver. However, the formulation is fully compatible with optimal-complexity solvers.} 

Let us stress that it is also possible to define a simpler preconditioner (motivated by \cite{baerland2020observation}) for the linearised problem \eqref{eq:newton} by using a Schur complement approach (following, e.g., \cite{mardal2011preconditioning, hong2017parameter,vassilevski2013block}) involving the space 
\[\underline{\underline{\bV}}\times \underline{\underline{\rQ}}:=(\kappa^{-1}+\rF\mathcal{H}(\revone{\bu^m}))[\bL^2(\Omega)\cap\bH_{\Gamma_{\bu}}(\vdiv,\Omega)] \times (\kappa^{-1}+\rF\mathcal{H}(\revone{\bu^m}))^{-1}\,\rL^2(\Omega),\]
where again  \((\revone{\bu^m},\revone{p^m})\in\underline{\underline{\bV}}\times \underline{\underline{\rQ}}\) is the state around which the Newton--Raphson linearisation is performed. 
This space is associated with the block diagonal Riesz map  
\begin{align}\label{precond.linear.new}\begin{pmatrix}
        \mathcal{P}:=\bigl([\kappa^{-1}\mathcal{I}+\rF 
        {\mathcal{H}(\revone{\bu^m})}]\, ({\mathcal{I}}-\nabla\vdiv)\bigr)^{-1} & 0\\0 &\!\!\!\! \!\!\!\!\!\!\!\mathcal{Q}:=([\kappa^{-1}\mathcal{I}+\rF 
        {\mathcal{H}(\revone{\bu^m})}]^{-1}\,\mathcal{I})^{-1}
    \end{pmatrix}.\end{align} 
The parameter-robustness for this preconditioner can be easily inferred from a scaling argument as in \cite[Example 3.2]{mardal2011preconditioning}. Nevertheless this preconditioner might not be of much interest in the case of coupling with other multiphysics effects. This has been discussed in length in \cite{baerland2020observation}. \revone{Note that while in the Darcy limit the operator preconditioning framework we use coincides with that of \cite{baerland2020observation}, unfortunately we do not see a clear way to continuously extend our construction  to include the case of zero Forchheimer coefficient. This is because the scaling with negative powers of this parameter are hard-coded in the norm and preconditioning terms, and the linear algebra breaks down for $\rF=0$. However, we can take positive, arbitrarily small values of this parameter (see, for instance, the behaviour for $\rF=10^{-9}$ in Example 4 of Section~\ref{sec:numerics}).}

\section{Numerical results}\label{sec:numerics}
The purpose of this section is to experimentally validate the theoretical results presented in the previous sections. We verify that the convergence rates are robust in the proposed finite element method \revone{in 2D and 3D}, and also illustrate the performance of proposed schemes in typical porous media flow problems. \revthree{The numerical implementation is based on the open-source finite element framework \texttt{Gridap} \cite{badia22}, and the main routines can be publicly accessed from \url{https://github.com/ruizbaier/GridapRobustDarcyForchheimer.jl}. }

\begin{table}[!t]
\setlength{\tabcolsep}{2.5pt}
\begin{center}
\revone{\footnotesize\begin{tabular}{|rccccccc|}
\hline
DoF  &    $h$ & $\|\bu-\bu_h\|_{3,\vdiv,\Omega}\!$  &  \texttt{rate} &  $\|p-p_h\|_{0,\Omega}$  &  \texttt{rate} & $\|P_h(\vdiv\bu_h-g)\|_{\ell^\infty}$ &  \texttt{it} \\
\hline
\multicolumn{8}{|c|}{Errors and convergence rates for $k = 0$}\\
\hline
   144 & 0.6124 & 9.78e-01 & $\star$ & 1.80e-01 & $\star$ & 1.40e-15 & 4\\
   1152 & 0.3062 & 5.30e-01 & 0.884 & 9.61e-02 & 0.907 & 2.24e-15 & 5\\
   9216 & 0.1531 & 2.72e-01 & 0.965 & 4.88e-02 & 0.977 & 6.53e-15 & 5\\
  73728 & 0.0765 & 1.37e-01 & 0.991 & 2.45e-02 & 0.994 & 1.59e-14 & 5\\
 589824 & 0.0383 & 6.85e-02 & 0.998 & 1.23e-02 & 0.999 & 1.64e-13 & 6\\
 4096512 & 0.0191 & 3.43e-02 & 0.999 & 6.15e-03 & 0.999 & 1.31e-13 & 6\\
\hline
\multicolumn{8}{|c|}{Errors and convergence rates for $k = 1$}\\
\hline
  624 & 0.6124 & 3.53e-01 & $\star$ & 6.32e-02 & $\star$ & 1.48e-14& 5 \\
   4992 & 0.3062 & 9.86e-02 & 1.840 & 1.73e-02 & 1.871 & 3.33e-14 & 6 \\
  39936 & 0.1531 & 2.55e-02 & 1.950 & 4.42e-03 & 1.968 & 9.14e-14 & 6\\
 319488 & 0.0765 & 6.45e-03 & 1.986 & 1.11e-03 & 1.992 & 2.25e-13 & 6\\
 2555904 & 0.0383 & 1.62e-03 & 1.993 & 2.78e-04 & 1.996 & 1.01e-13 & 7 \\
 20447232 & 0.0191 & 4.07e-04 & 1.997 & 6.95e-05 & 1.998 & 2.01e-13 & 7\\
 \hline
\end{tabular}}
\end{center}

\vspace{0.1cm}
\caption{Example 1. Error history against smooth manufactured solutions \revone{in 3D} (errors on a sequence of successively refined grids, convergence rates, 
and norm of the divergence of the discrete velocity) for different polynomial degrees, and iteration count for the nonlinear Newton--Raphson solver. Here we use unity parameters (and $r=3$) and the $\bH^{3}(\vdiv,\Omega)$ and $\rL^2(\Omega)$ norms for velocity and pressure, respectively. } \label{table:unity parameters}
\end{table}

\begin{figure}[!t]
 \begin{center}
  \includegraphics[width=0.485\textwidth]{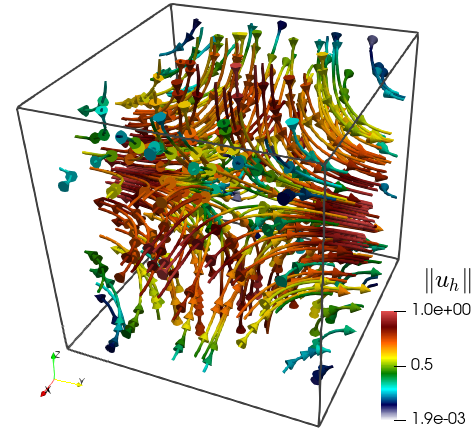}
  \includegraphics[width=0.485\textwidth]{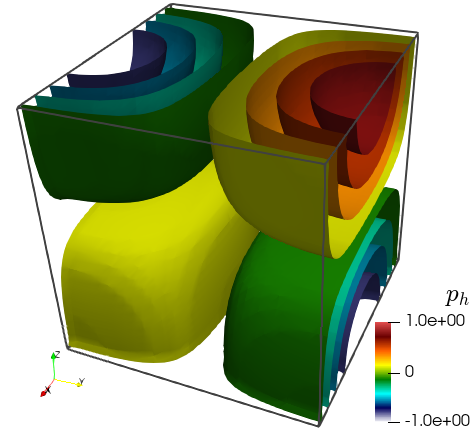}
 \end{center}
  
\vspace{-0.3cm}
\caption{\revtwo{Example 1. Streamlines of the approximate velocity (left) and iso-contours of pressure distribution (right) for the accuracy test of the Darcy--Forchheimer equations in 3D with unit parameters.}}\label{fig:sample01}
\end{figure}

\medskip 
\noindent\textbf{Example 1 (convergence against smooth solutions \revtwo{in 3D}).} Consider the unit \revtwo{cube  domain \(\Omega = (0,1)^3\) along with the following manufactured solutions
\begin{gather*}
 p(x, y,z) := \sin(\pi x)\cos(\pi y)\sin(\pi z),\quad 
\bu(x,y,z) := \begin{pmatrix}
\cos(\pi x)\sin(\pi y)\sin(\pi z)\\
-\sin(\pi x)\cos(\pi y)\sin(\pi z)\\
\sin(\pi x)\sin(\pi y) \cos(\pi z)\end{pmatrix}.\end{gather*}
We set the left, bottom, and front sides of the cube as $\Gamma_{\bu}$ and the right, top, and back} sides as $\Gamma_p$. The forcing vector $\bf$ and the fluid source $g$, as well as the (non-homogeneous) essential boundary condition on the flux, are computed based on the manufactured solutions above. We take, for this first set of examples, simply dimensional unity parameters $\kappa = \rF = 1$ and index $r = 3$. 
The error history associated with the 
proposed mixed finite element  method on a sequence of successively refined partitions of the domain, are presented in Table~\ref{table:unity parameters}, where we also tabulate rates of error decay computed as 
$\texttt{rate}  =\log(e_{(\bullet)}/\tilde{e}_{(\bullet)})[\log(h/\tilde{h})]^{-1}$,
where $e,\tilde{e}$ denote errors generated on two
consecutive  meshes of sizes $h$ and~$\tilde{h}$, respectively. The results indicate optimal convergence in all fields and for the two tested polynomial degrees. \revtwo{Moreover, for this problem  the number of  Newton--Raphson iterations required to reach a tolerance (either absolute or relative) of $10^{-8}$ is up to seven. This depends on the initial guess,  and here we have used a constant velocity and pressure such that every entry of the vector of degrees of freedom is $10^{-4}$}. Sample approximate solutions for velocity and pressure obtained with the second-order method (with $k=1$) are plotted in Figure~\ref{fig:sample01}. For these results, we consider the norms $\bH^{3}(\vdiv,\Omega)$ and $\rL^2(\Omega)$ for velocity and pressure (with no parameter weighting).

\begin{table}[!t]
\setlength{\tabcolsep}{4.4pt}
\begin{center}
\revtwo{\footnotesize\begin{tabular}{|rccccccc|}
\hline
DoF  &    $h$ & $\|\bu-\bu_h\|_{3,\vdiv,\Omega}\!$  &  \texttt{rate} &  $\|p-p_h\|_{0,\Omega}$  &  \texttt{rate} & $\|P_h(\vdiv\bu_h-g)\|_{\ell^\infty}$ & \texttt{it} \\
\hline
\multicolumn{8}{|c|}{Errors and convergence rates for $k = 0$}\\
\hline
     36 & 0.7071 & 5.18e-01 & $\star$ & 1.01e+09 & $\star$ & 2.22e-16  & 4 \\
     84 & 0.3536 & 4.25e-01 & 0.285 & 3.88e+08 & 1.385 & 1.33e-15 & 6 \\
    240 & 0.1768 & 4.56e-01 & -0.101 & 1.78e+08 & 1.127 & 1.35e-15 & 6 \\
    792 & 0.0884 & 2.77e-01 & 0.718 & 6.29e+07 & 1.499 & 4.00e-15 & 5 \\
   2856 & 0.0442 & 1.53e-01 & 0.857 & 1.93e+07 & 1.707 & 9.32e-15 & 5 \\
  10824 & 0.0221 & 8.29e-02 & 0.884 & 6.53e+06 & 1.561 & 2.20e-14 & 5 \\
  \hline
\multicolumn{8}{|c|}{Errors and convergence rates for $k = 1$}\\
\hline
  120 & 0.7071 & 4.07e-01 & $\star$ & 1.18e+08 & $\star$ & 3.02e-15 & 6\\ 
    276 & 0.3536 & 3.70e-01 & 0.138 & 6.18e+07 & 0.929 & 8.07e-15 & 6 \\
    780 & 0.1768 & 2.47e-01 & 0.582 & 2.29e+07 & 1.434 & 1.39e-14  & 6 \\
   2556 & 0.0884 & 1.13e-01 & 1.136 & 7.71e+06 & 1.568 & 5.32e-14  & 7 \\
   9180 & 0.0442 & 3.74e-02 & 1.589 & 2.34e+06 & 1.724 & 2.06e-13  & 6 \\
  34716 & 0.0221 & 1.67e-02 & 1.159 & 6.33e+05 & 1.883 & 4.90e-13 & 7 \\
 \hline
 \hline
DoF  &    $h$ & $\|\bu-\bu_h\|_{\bV}\!$  &  \texttt{rate} &  $\|p-p_h\|_{\widehat{\rQ}(\cT_h)\vphantom{\int_{X_x}}}$  &  \texttt{rate} & $\|P_h(\vdiv\bu_h-g)\|_{\ell^\infty}$  \\
\hline
\multicolumn{8}{|c|}{Errors and convergence rates for $k = 0$}\\
\hline
     36 & 0.7071 & 1.44e+03 & $\star$ & 7.16e+02 & $\star$ & 2.22e-16 & 4 \\
     84 & 0.3536 & 5.95e+02 & 1.276 & 2.67e+02 & 1.424 & 6.66e-16 & 6\\
    240 & 0.1768 & 3.65e+02 & 0.705 & 1.14e+02 & 1.221 & 1.18e-15 & 6 \\
    792 & 0.0884 & 1.64e+02 & 1.153 & 4.08e+01 & 1.490 & 3.00e-15 & 5\\
   2856 & 0.0442 & 5.26e+01 & 1.642 & 2.16e+01 & 1.352 & 7.55e-15 & 5\\
  10824 & 0.0221 & 2.18e+01 & 1.272 & 1.04e+00 & 1.203 & 1.47e-14 & 5\\
\hline
\multicolumn{8}{|c|}{Errors and convergence rates for $k = 1$}\\
\hline
   120 & 0.7071 & 4.64e+02 & $\star$ & 1.66e+02 & $\star$ & 3.06e-15 & 6 \\
    276 & 0.3536 & 2.25e+02 & 1.044 & 3.62e+01 & 2.194 & 4.72e-15 & 6\\
    780 & 0.1768 & 1.11e+02 & 1.017 & 1.29e+01 & 1.484 & 1.28e-14 & 6 \\
   2556 & 0.0884 & 3.41e+01 & 1.706 & 5.59e+00 & 1.211 & 4.45e-14 & 7 \\
   9180 & 0.0442 & 7.01e+00 & 2.282 & 2.56e+00 & 1.129 & 6.74e-14 & 6 \\
  34716 & 0.0221 & 1.16e+00 & 2.601 & 6.80e-01 & 1.912 & 1.34e-13  & 7 \\
 \hline
\end{tabular}}
\end{center}

\vspace{0.1cm}
\caption{Example 2. Error history against a fine-mesh reference solution (experimental errors on a sequence of successively refined grids, convergence rates,  norm of the divergence of the discrete velocity, \revtwo{and Newton iteration count}) for different polynomial degrees. Here we use $\kappa=10^{-8}$, $\rF = 10^4$ (and $r=3$), and the non-weighted (top rows) and parameter-weighted (bottom rows) norms for velocity and pressure.} \label{table:extreme parameters}
\end{table}

\medskip
\noindent\textbf{Example 2 (experimental self-convergence with extreme parameters).}  For a second run of simulations we use the domain $\Omega=(0,2)\times(0,1)$ without a known analytical solution, setting the problem data as $\bf = \bzero$, $g = 0$, mixed boundary conditions as follows $\bu \cdot \bn = 2.5y(1-y)$ on the left segment, $\bu\cdot\bn = 0$ on the top and bottom segments, and $p=0$ on the right edge. The permeability is heterogeneous $\kappa(\bx)=\kappa_0(1+ \exp(-\frac12[10y-5-\sin(10x)]^2)$, and we consider the more \revtwo{pronounced} parameter values $\kappa_0 = 10^{-8}$, $\rF = 10^{4}$, and $r =3$. In Table~\ref{table:extreme parameters} we show the error decay where we compute experimental errors against a  fine-mesh reference solution $\bu_{\mathrm{ref}},p_{\mathrm{ref}}$(with two more refinement steps than the finest level). In the first part of the table we use the non-weighted norms and we see very large errors in the pressure \revtwo{and} a sub-optimal error decay for the velocity in the second-order case.  In contrast, for the second part of the table the errors are now measured in the weighted norms associated with $\bV$ and $\widehat{\rQ}(\cT_h)$ defined in \eqref{eq:discrete.Vh} and \eqref{eq:broken.Qhat}, respectively, where we move the scaling of the now variable permeability inside the norm definition. {We recall from \cite{baerland2020observation} that in the space $\rQ_h$ the realisation of the required discrete weighted Laplacian operator that is the Riesz map for  $\kappa^{\frac12}\rH^1_{h,\Gamma_p}(\Omega)$ is as follows
\begin{align*} \langle -\vdiv(\kappa(\bx) \nabla p_h), q_h\rangle & = \sum_{K\in \cT_h} (\kappa(\bx) \nabla p_h, \nabla q_h)_K 
+ \sum_{F\in \cE_h\cup \cE_h^{\Gamma_p}} \frac{1}{h_F}(\kappa(\bx)[\![ p_h ]\!],[\![ q_h ]\!])_F.\end{align*}
Here $\cE_h$ denotes the set of internal facets, $\cE_h^{\Gamma_p}$ denotes the set of facets lying on the sub-boundary $\Gamma_p$, and 
$[\![ \bullet ]\!]$ represents the 
jump operator across a facet (and adopt the convention that it reduces to the identity if the face lies on the boundary).  In turn, for the space $\rF^{-\frac13}\rW^{1,\frac32}_{h,\Gamma_p}(\Omega)$  the natural operator (duality map) is the nonlinear discrete $r'-$Laplacian
\begin{align*} \langle -\vdiv (\rF^{r-4}|\nabla p_h|^{r'-2}\nabla p_h),q_h\rangle &= \sum_{K\in \cT_h} \rF^{r-4} (|\nabla p_h|^{r'-2}
\nabla p_h, \nabla q_h)_K \\ &\quad + \sum_{F\in \cE_h\cup \cE_h^{\Gamma_p}}\!\!\! \rF^{r-4} h_F^{1-r'}(
|[\![ p_h ]\!]|^{r'-2}
[\![ p_h ]\!],[\![ q_h ]\!])_F,  \end{align*}
with $r' = \frac32$ (see, e.g., \cite{burman2008discontinuous}). 
However, since we require a linear operator to define the preconditioner-induced norm, we use instead the following linearisation around the reference pressure 
\begin{align*}
& \sum_{K\in \cT_h} \rF^{r-4} (|\nabla p_{\mathrm{ref}}|^{r'-2}
\nabla p_h, \nabla q_h)_K + \sum_{F\in \cE_h\cup \cE_h^{\Gamma_p}} \!\!\!\!\!\rF^{r-4}(r'-1)h_F^{1-r'}(
|[\![ p_{\mathrm{ref}} ]\!]|^{r'-2}
[\![ p_h ]\!],[\![ q_h ]\!])_F \\
&\quad + \sum_{K\in \cT_h} \rF^{r-4}(r'-2) (|\nabla p_{\mathrm{ref}}|^{r'-4}
(\nabla p_{\mathrm{ref}}\cdot\nabla p_h),(\nabla p_{\mathrm{ref}} \cdot \nabla q_h))_K.  
\end{align*}
It is observed that the rates remain optimal, as anticipated by Theorem~\ref{th:convergence}. Snapshots of the numerical solutions for the lowest-order scheme are shown in Figure~\ref{fig:example02}.

\begin{figure}[t!]
    \centering
    \includegraphics[width=0.325\linewidth]{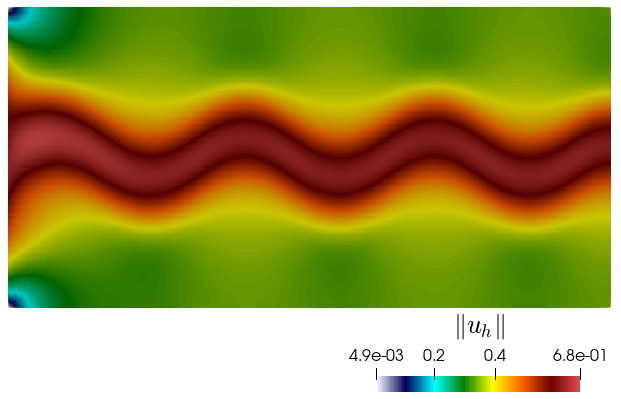}
    \includegraphics[width=0.325\linewidth]{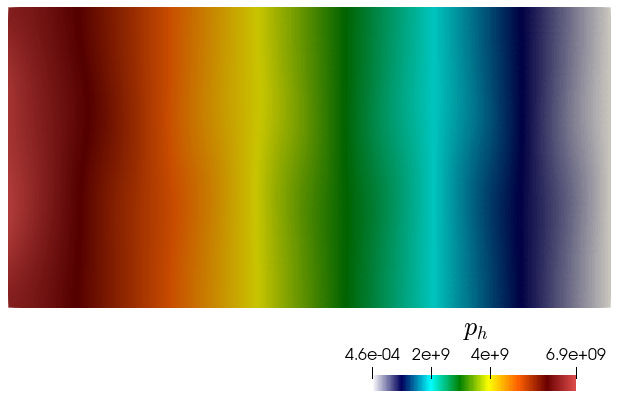}
    \includegraphics[width=0.325\linewidth]{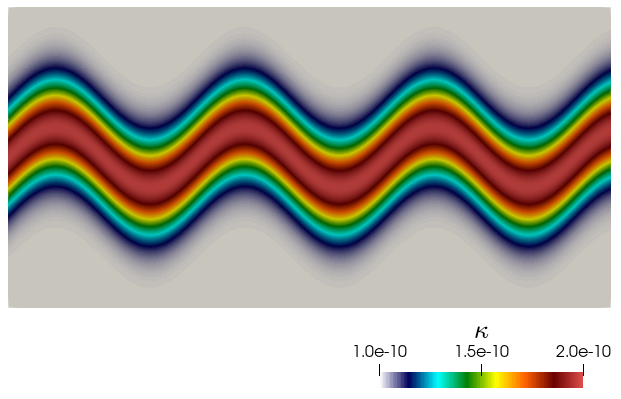}
    \caption{Example 2. Sample of filtration velocity magnitude, pressure profile, and heterogeneous permeability, for the convergence test with no known analytical solution.}
    \label{fig:example02}
\end{figure}

\medskip 
\noindent\textbf{Example 3 (application to filtration in highly permeable media).} In this test example, we simulate typical flow through a porous domain with large obstacles. The domain has length of 0.1\,m and includes seven circular cylinders of different sizes. The permeability is not isotropic, but a tensor (that also needs to be scaled with fluid viscosity). In its principal direction frame, we have 
\begin{equation} \bkappa_0 = \frac{1}{\nu}\begin{pmatrix} \kappa_1 & 0 \\ 0 & \kappa_2\end{pmatrix},\label{rescaled-kappa}\end{equation}
with $\nu = 10^{-6}$\,m$^2$s$^{-1}$, $\kappa_1 = 5\cdot10^{-10}$\,m$^2$ and $\kappa_2 = 10^{-10}$\,m$^2$, and the tensor is rotated by $4.7^\circ =  0.082$\,rad from the $x-$axis giving,  
\[ \bkappa  = \bR \bkappa_0 \bR^{\tt t}, \quad \text{where} \quad \bR = \begin{pmatrix} \cos(\theta) & -\sin(\theta) \\ \sin(\theta) &  \cos(\theta)\end{pmatrix}.\]       
The external force $\bf$ has magnitude $1.3\cdot 10^{-5}$\,ms$^{-2}$ and it points in the $x$ direction, and we assume zero fluid source $g=0$. We consider an inlet boundary on the bottom left circular arc, on which we impose a radial inlet velocity $\bu \cdot \bn = \frac{1}{4}\bx/|\bx|$\,m/s; an outlet boundary on the upper right circular arc, on which we impose $p = 0$, and on the remainder of the boundary (including the obstacles) we set $\bu\cdot \bn = 0$. Parameter values and flow configurations have been adapted from  \cite{hassard2023comparison}.
The unstructured triangular mesh has 26,184 elements and we use for this test a second-order scheme (setting $k=1$), giving a total number of 197,296 DoFs. The remaining parameters are the Forchheimer coefficient and index, taken here as $\rF = 1$\,m$^{-1} \cdot \rho$ with $\rho = 10^3$\,kg\,m$^{-3}$, and $r=3$. Finally, we note that the use of a Forchheimer correction term is justified in this case since the Forchheimer threshold coefficient (recall that the permeability coefficient is already scaled with viscosity in \eqref{rescaled-kappa}) 
\[ \rF_0 = 
\kappa_{\max}
\rF  |\bu|\]
is 0.245 (the rule suggested in \cite{zeng2006criterion}, see also \cite{fumagalli2022model}, specifies that it has to be larger than 0.1) when taking for $|\bu|$ its maximum value over $\Omega$ (not known a priori but computed to be 0.49). 
We depict in Figure~\ref{fig:sample03} the approximate solutions, indicating the expected filtration patterns in porous media with large obstacles. \revtwo{For this set of simulations the Newton--Raphson solver took nine iterations to converge, and we stress that the iteration count is dependent on the initial guess and parameter configuration.}

\begin{figure}[!t]
 \begin{center}
  \includegraphics[width=0.485\textwidth]{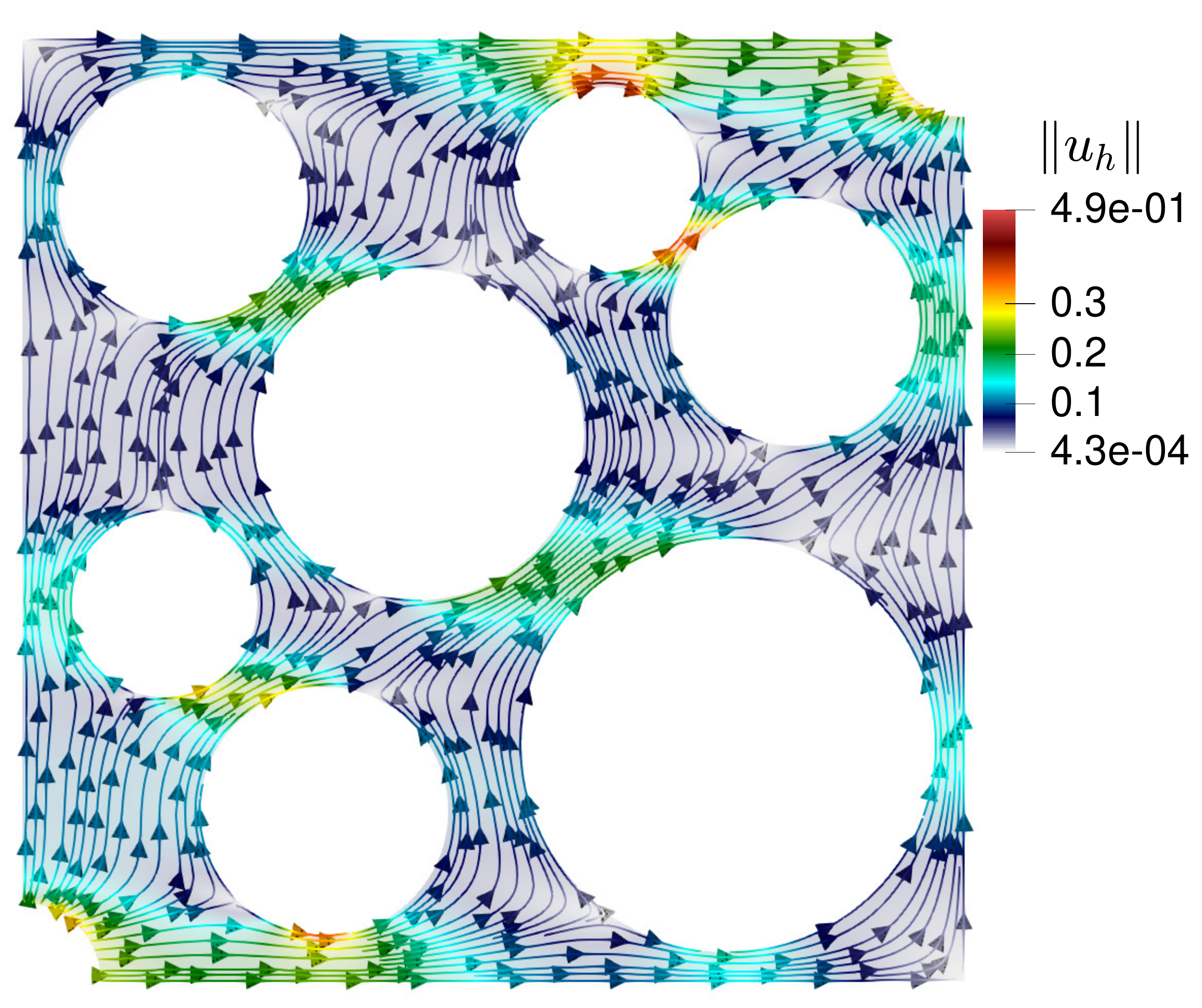}
  \includegraphics[width=0.485\textwidth]{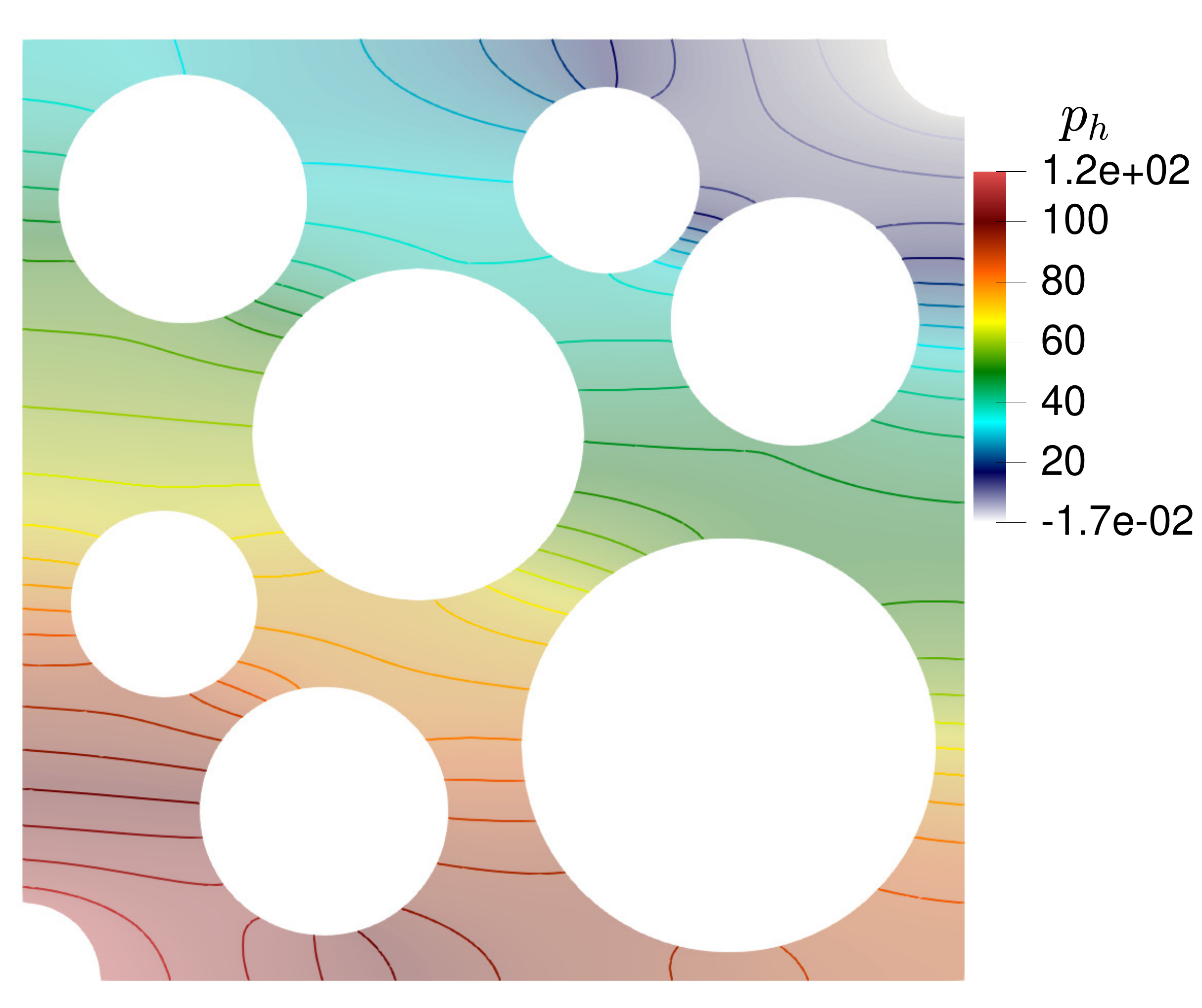}
 \end{center}
  
\vspace{-0.3cm}
\caption{Example 3. Streamlines of the approximate velocity coloured by velocity magnitude (left) and contours of pressure distribution (right), for the flow across obstacles using the Darcy--Forchheimer equations with anisotropic permeability.  A second-order scheme ($k=1$) is used.}\label{fig:sample03}
\end{figure}

\medskip 
\textbf{Example 4 (operator preconditioning for the linearised problem).} To conclude we conduct a series of computations where the  permeability $\kappa$ \revtwo{is varied over several orders of magnitude, as well as the Forchheimer coefficient $\rF$. We also take three values of the index $r$}.  This is done again using the domain $\Omega = (0,1)^2$, the lowest-order finite element family, and mixed boundary conditions as in Example 1. The performance of the numerical method is assessed by comparing an estimate of the condition number of the  linearised system matrix \eqref{eq:newton} preconditioned with the weighted duality map \eqref{precond.linear}, \revtwo{and also inspect the number of MinRes iterations required to solve the preconditioned system}. Table~\ref{tab:precond} illustrates the robustness of the proposed preconditioner across the parametric space. Additionally, the performance of the preconditioner corresponding to the weighted Riesz map \eqref{precond.linear.new} has also been illustrated in Table~\ref{tab:precond.new}, which seems to perform slightly better than \eqref{precond.linear}. This is a behaviour very similar  to the one reported in \cite{baerland2020observation}, but the map \eqref{precond.linear}  is  preferable in a multiphysics context where the scaling of the divergence component of the velocity block  is not feasible. \revtwo{Note also that both preconditioners are perfectly robust on $r$ (the three sub-tables on each table look almost identical in both condition number and MinRes iteration count). They are also robust on $h$ (the growth in the rows is minimal). However, for the preconditioner induced by \eqref{precond.linear}, there are some effects of small values of $\kappa$, which are more noticeable in the MinRes iterations than in the condition number. It is evident from Table~\ref{tab:precond.new} that large values of $\rF$ only seem to affect the number of iterations when $\kappa$ is large too, but in applications one typically expects large $\rF$ associated with small  $\kappa$.} 

\begin{table}[t!]
\setlength{\tabcolsep}{2.5pt}
\begin{center}
\revtwo{\scriptsize\begin{tabular}{|c|c|c|lll|}
  \hline
  \multirow{2}{*}{$r$} & \multirow{2}{*}{$\kappa$} & \multirow{2}{*}{$\rF$} & \multicolumn{3}{c|}{$h$}\\
  \cline{4-6}
  & & & {$\sqrt{2}^{-2}$ } & {$\sqrt{2}^{-3}$} & {$\sqrt{2}^{-4}$} \vphantom{$\int^{\int^X}$}\\
\hline  
  \multirow{9}{*}{3} 
  & \multirow{3}{*}{$10^{-9}$} \vphantom{$\int^{X}$}& {$10^{-9}$} 
                      & 3.74 (23) & 4.18 (31) & 4.41 (33) \\  
&  & {1}          & 3.74 (23) & 4.18 (31) & 4.41 (33)\\
&  & {$10^4$} & 3.74 (23) & 4.18 (31) & 4.41 (33)\\
  \cline{2-6}
&\multirow{3}{*}{$10^{-4}$} & {$10^{-9}$} \vphantom{$\int^{X}$} 
                       & 3.65 (25) & 3.82 (29) & 3.57 (30)\\ 
&  & {1}           & 3.65 (25) & 3.82 (29) & 3.57 (30) \\
&  & {$10^4$} & 3.65 (25) & 3.82 (29) & 3.59 (30) \\
\cline{2-6}  
&\multirow{3}{*}{1} & {$10^{-9}$} \vphantom{$\int^{X}$} 
                       & 1.23 (11) & 1.21 (11) & 1.21 (11)\\
&  & {1}           & 1.24 (11) & 1.23 (11) & 1.22 (11) \\      
&  & {$10^4$} & 4.09 (27) & 4.07 (33) & 4.06 (35)\\
\hline
\end{tabular}
\quad 
\begin{tabular}{|c|c|c|lll|}
  \hline
  \multirow{2}{*}{$r$} & \multirow{2}{*}{$\kappa$} & \multirow{2}{*}{$\rF$} & \multicolumn{3}{c|}{$h$}\\
  \cline{4-6}
  & & & {$\sqrt{2}^{-2}$ } & {$\sqrt{2}^{-3}$} & {$\sqrt{2}^{-4}$} \vphantom{$\int^{\int^X}$}\\
\hline  
  \multirow{9}{*}{3.5} 
  & \multirow{3}{*}{$10^{-9}$} \vphantom{$\int^{X}$}& {$10^{-9}$} 
                      & 3.74 (23) & 4.18 (31) & 4.41 (33) \\  
&  & {1}          & 3.74 (23) & 4.18 (31) & 4.41 (33)\\
&  & {$10^4$} & 3.74 (23) & 4.18 (31) & 4.41 (33)\\
  \cline{2-6}
&\multirow{3}{*}{$10^{-4}$} & {$10^{-9}$} \vphantom{$\int^{X}$} 
                       & 3.65 (25) & 3.82 (29) & 3.57 (30)\\ 
&  & {1}           & 3.65 (25) & 3.82 (29) & 3.57 (30) \\
&  & {$10^4$} & 3.67 (25) & 3.85 (29) & 3.58 (30) \\
\cline{2-6}  
&\multirow{3}{*}{1} & {$10^{-9}$} \vphantom{$\int^{X}$} 
                       & 1.23 (11) & 1.21 (11) & 1.21 (11)\\
&  & {1}           & 1.23 (11) & 1.22 (11) & 1.22 (11) \\      
&  & {$10^4$} & 4.58 (31) & 3.84 (31) & 3.75 (32)\\
\hline
\end{tabular}
\medskip 

\begin{tabular}{|c|c|c|lll|}
  \hline
  \multirow{2}{*}{$r$} & \multirow{2}{*}{$\kappa$} & \multirow{2}{*}{$\rF$} & \multicolumn{3}{c|}{$h$}\\
  \cline{4-6}
  & & & {$\sqrt{2}^{-2}$ } & {$\sqrt{2}^{-3}$} & {$\sqrt{2}^{-4}$} \vphantom{$\int^{\int^X}$}\\
\hline  
  \multirow{9}{*}{4} 
  & \multirow{3}{*}{$10^{-9}$} \vphantom{$\int^{X}$}& {$10^{-9}$} 
                      & 3.74 (23) & 4.18 (31) & 4.41 (33) \\  
&  & {1}          & 3.74 (23) & 4.18 (31) & 4.41 (33)\\
&  & {$10^4$} & 3.74 (23) & 4.18 (31) & 4.41 (33)\\
  \cline{2-6}
&\multirow{3}{*}{$10^{-4}$} & {$10^{-9}$} \vphantom{$\int^{X}$} 
                       & 3.65 (25) & 3.82 (29) & 3.57 (30)\\ 
&  & {1}           & 3.65 (25) & 3.82 (29) & 3.57 (30) \\
&  & {$10^4$} & 3.66 (25) & 3.84 (29) & 3.57 (30) \\
\cline{2-6}  
&\multirow{3}{*}{1} & {$10^{-9}$} \vphantom{$\int^{X}$} 
                       & 1.23 (11) & 1.21 (11) & 1.21 (11)\\
&  & {1}           & 1.23 (11) & 1.22 (11) & 1.22 (11) \\      
&  & {$10^4$} & 3.84 (27) & 3.16 (26) & 3.05 (26)\\
\hline
\end{tabular}}
\end{center}
\vspace{3mm}
\caption{Example 4. Estimated condition number \revtwo{and MinRes iterations required to converge, associated with} the preconditioned system matrix of the linearised Darcy--Forchheimer equations using the preconditioner induced by \eqref{precond.linear}.}
    \label{tab:precond}

\end{table}


\begin{table}[t!]
\setlength{\tabcolsep}{2.5pt}
\begin{center}
\revtwo{\scriptsize\begin{tabular}{|c|c|c|lll|}
  \hline
  \multirow{2}{*}{$r$} & \multirow{2}{*}{$\kappa$} & \multirow{2}{*}{$\rF$} & \multicolumn{3}{c|}{$h$}\\
  \cline{4-6}
  & & & {$\sqrt{2}^{-2}$ } & {$\sqrt{2}^{-3}$} & {$\sqrt{2}^{-4}$} \vphantom{$\int^{\int^X}$}\\
\hline  
  \multirow{9}{*}{3} 
  & \multirow{3}{*}{$10^{-9}$} \vphantom{$\int^{X}$}& {$10^{-9}$} 
                      & 1.20 (5) & 1.20 (1) & 1.20 (1) \\  
&  & {1}          & 1.20 (5) & 1.20 (1) & 1.20 (1)\\
&  & {$10^4$} & 1.20 (5) & 1.20 (1) & 1.20 (1)\\
  \cline{2-6}
&\multirow{3}{*}{$10^{-4}$} & {$10^{-9}$} \vphantom{$\int^{X}$} 
                       & 1.20 (6) & 1.20 (6) & 1.20 (6)\\ 
&  & {1}           & 1.20 (6) & 1.20 (6) & 1.20 (6) \\
&  & {$10^4$} & 1.20 (8) & 1.20 (8) & 1.20 (8) \\
\cline{2-6}  
&\multirow{3}{*}{1} & {$10^{-9}$} \vphantom{$\int^{X}$} 
                       & 1.20 (6) & 1.21 (6) & 1.20 (6)\\
&  & {1}           & 1.20 (8) & 1.20 (8) & 1.20 (8) \\      
&  & {$10^4$} & 1.39 (13) & 1.36 (13) & 1.32 (11)\\
\hline
\end{tabular}
\quad 
\begin{tabular}{|c|c|c|lll|}
  \hline
  \multirow{2}{*}{$r$} & \multirow{2}{*}{$\kappa$} & \multirow{2}{*}{$\rF$} & \multicolumn{3}{c|}{$h$}\\
  \cline{4-6}
  & & & {$\sqrt{2}^{-2}$ } & {$\sqrt{2}^{-3}$} & {$\sqrt{2}^{-4}$} \vphantom{$\int^{\int^X}$}\\
\hline  
  \multirow{9}{*}{3.5} 
  & \multirow{3}{*}{$10^{-9}$} \vphantom{$\int^{X}$}& {$10^{-9}$} 
                      & 1.20 (5) & 1.20 (1) & 1.20 (1) \\  
&  & {1}          & 1.20 (5) & 1.20 (1) & 1.20 (1)\\
&  & {$10^4$} & 1.20 (5) & 1.20 (1) & 1.20 (1)\\
  \cline{2-6}
&\multirow{3}{*}{$10^{-4}$} & {$10^{-9}$} \vphantom{$\int^{X}$} 
                       & 1.20 (6) & 1.20 (6) & 1.20 (6)\\ 
&  & {1}           & 1.20 (6) & 1.20 (6) & 1.20 (6) \\
&  & {$10^4$} & 1.20 (8) & 1.20 (8) & 1.20 (8) \\
\cline{2-6}  
&\multirow{3}{*}{1} & {$10^{-9}$} \vphantom{$\int^{X}$} 
                       & 1.20 (6) & 1.21 (6) & 1.20 (6)\\
&  & {1}           & 1.20 (8) & 1.20 (8) & 1.20 (8) \\      
&  & {$10^4$} & 1.69 (17) & 1.62 (15) & 1.53 (14)\\
\hline
\end{tabular}

\medskip 
\begin{tabular}{|c|c|c|lll|}
  \hline
  \multirow{2}{*}{$r$} & \multirow{2}{*}{$\kappa$} & \multirow{2}{*}{$\rF$} & \multicolumn{3}{c|}{$h$}\\
  \cline{4-6}
  & & & {$\sqrt{2}^{-2}$ } & {$\sqrt{2}^{-3}$} & {$\sqrt{2}^{-4}$} \vphantom{$\int^{\int^X}$}\\
\hline  
  \multirow{9}{*}{4} 
  & \multirow{3}{*}{$10^{-9}$} \vphantom{$\int^{X}$}& {$10^{-9}$} 
                      & 1.20 (5) & 1.20 (1) & 1.20 (1) \\  
&  & {1}          & 1.20 (5) & 1.20 (1) & 1.20 (1)\\
&  & {$10^4$} & 1.20 (5) & 1.20 (1) & 1.20 (1)\\
  \cline{2-6}
&\multirow{3}{*}{$10^{-4}$} & {$10^{-9}$} \vphantom{$\int^{X}$} 
                       & 1.20 (6) & 1.20 (6) & 1.20 (6)\\ 
&  & {1}           & 1.20 (6) & 1.20 (6) & 1.20 (6) \\
&  & {$10^4$} & 1.20 (7) & 1.20 (7) & 1.20 (7) \\
\cline{2-6}  
&\multirow{3}{*}{1} & {$10^{-9}$} \vphantom{$\int^{X}$} 
                       & 1.20 (6) & 1.21 (6) & 1.20 (6)\\
&  & {1}           & 1.20 (7) & 1.20 (7) & 1.20 (7) \\      
&  & {$10^4$} & 1.98 (19) & 1.73 (17) & 1.52 (13)\\
\hline
\end{tabular}}
\end{center}
\vspace{3mm}
\caption{Example 4. Estimated condition number \revtwo{and MinRes iterations required to converge, associated with}  the preconditioned system matrix of the linearised Darcy--Forchheimer equations using the preconditioner induced by \eqref{precond.linear.new}.}
    \label{tab:precond.new}
    
\end{table}

\revtwo{We stress that the inverse of each block in the preconditioner is applied in an abstract operator sense. In particular, we use the LinearOperator abstraction in Gridap/Julia to define the (linearised) weak form associated with each block of the preconditioner.
This allows us to apply a block-diagonal approximation of the preconditioning without explicitly assembling or inverting the full operator. In the numerical examples presented, each block is inverted using the sparse direct solver UMFPACK. We have not employed multigrid solvers, as the numerical examples are of moderate size and are intended to validate the theoretical framework rather than to optimise performance. Nevertheless, the proposed preconditioners are fully compatible with multigrid-based approximations, and we expect the robustness with respect to 
$\kappa$ and $\rF$ to be retained under standard spectral equivalence assumptions.}

    \bibliographystyle{siam}
\bibliography{refs}
\end{document}